\documentclass[leqno,english]{amsart}
\usepackage{amsfonts,amssymb,amsmath,amsgen,amsthm}
\usepackage{hyperref}
\newcommand{\msc}[2][2000]{%
  \let\@oldtitle\@title%
  \gdef\@title{\@oldtitle\footnotetext{#1 \emph{Mathematics subject
        classification.} #2}}%
}

\theoremstyle{plain}
\newtheorem{theorem}{Theorem} [section]

\newtheorem{lemma}[theorem]{Lemma}
\newtheorem{corollary}[theorem]{Corollary}
\newtheorem{proposition}[theorem]{Proposition}

\theoremstyle{remark}
\newtheorem{remark}[theorem]{Remark}

\def\le{\leqslant}
\def\ge{\geqslant}

\def\Eq#1#2{\mathop{\sim}\limits_{#1\rightarrow#2}}
\def\Tend#1#2{\mathop{\longrightarrow}\limits_{#1\rightarrow#2}}

\DeclareMathOperator{\RE}{Re}
\DeclareMathOperator{\IM}{Im}

\numberwithin{equation}{section}

\begin{document}

\title[Universal dynamics for the logarithmic NLS]{Universal  
dynamics for the  defocusing logarithmic Schr\"odinger equation} 

\author[R. Carles]{R\'emi Carles}
\address{CNRS, IMAG, Montpellier, France}
\email{Remi.Carles@math.cnrs.fr}

\author[I. Gallagher]{Isabelle Gallagher}
\address{D\'epartement de math\'ematiques et applications\\ \'Ecole normale
  sup\'erieure, CNRS, PSL Research University\\ 75005 Paris, France\\ \& 
UFR de math\'ematiques\\ Universit\'e Paris-Diderot, Sorbonne Paris-Cit\'e\\
75013 Paris, France} 
\email{Isabelle.Gallagher@ens.fr}

\begin{abstract}
  We consider the nonlinear Schr\"odinger equation with a logarithmic
  nonlinearity in a dispersive regime. We show that the presence of the nonlinearity
  affects the large time behavior of the solution: 
  the dispersion is faster than usual 
  by a logarithmic factor in time and the positive Sobolev norms of
  the solution 
  grow logarithmically in time. Moreover, after rescaling in space by the dispersion   rate, the modulus of the solution converges to a universal Gaussian
  profile. These properties are suggested by explicit computations in
  the case of Gaussian initial data, and remain when
    an extra power-like nonlinearity is present in the equation.
One of the key steps of the proof consists in working
  in hydrodynamical variables to reduce the equation to a variant of 
  the isothermal compressible Euler equation, whose large time
  behavior turns out to be governed by a parabolic equation involving a
   Fokker--Planck operator.
\end{abstract}
\thanks{RC is supported by the French ANR projects
   BECASIM
  (ANR-12-MONU-0007-04) and BoND (ANR-13-BS01-0009-01).}
\maketitle

\section{Introduction}
\label{sec:introduction}

\subsection{Setting}
\label{sec:setting}

We are interested in the following equation
\begin{equation}
  \label{eq:logNLS}
  i{\partial}_t u +\frac{1}{2} \Delta u =\lambda \ln\left(|u|^2\right)u \, ,\quad u_{\mid t=0} =u_0 \, ,
\end{equation}
with $x\in {\mathbb R}^d$, $d\ge 1$, $\lambda\in {\mathbb R}\setminus\{0\}$. 
It was introduced as a model of
nonlinear wave mechanics and in nonlinear optics (\cite{BiMy76}, see also
\cite{BSSSSC03,Hef85,HeRe80,KEB00,DFGL03}).
The mathematical study of this equation goes back to
\cite{CaHa80,Caz83} (see also 
\cite{CazCourant}).  The sign $\lambda<0$ seems
to be the more interesting from a physical point of view, and this
case has been
studied formally and rigorously (see~\cite{AMS14,HeRe80} for instance). On the other hand, the case
$\lambda>0$ seems to have been little studied mathematically, except
as far as the Cauchy problem is concerned (see
\cite{CaHa80,GLN10}). In this article, we address the large time
properties of the solution in the case $\lambda>0$, revealing several
new features in the context of Schr\"odinger equations, and more generally Hamiltonian dispersive equations. 
\medskip

\noindent 
We recall that   the mass, angular momentum and  energy are (formally)
conserved, in the sense that defining 
\begin{align*}
  & M(u(t)) :=\|u(t)\|_{L^2({\mathbb R}^d)}^2 \, , \\
& J(u(t)):=\IM \int_{{\mathbb R}^d}\bar u(t,x)\nabla u(t,x)dx\,,\\
& E(u(t)):=\frac{1}{2}\|\nabla
  u(t)\|_{L^2({\mathbb R}^d)}^2+\lambda\int_{{\mathbb R}^d} |u(t,x)|^2\ln |u(t,x)|^2dx \, , 
\end{align*}
then formally,
\begin{equation*}
  \frac{d}{dt}M(u(t))=\frac{d}{dt}J(u(t))=\frac{d}{dt}E(u(t))=0 \, .
  \end{equation*}
The last identity reveals the Hamiltonian structure of
\eqref{eq:logNLS}. 

\begin{remark}[Effect of scaling factors]\label{rem:scaling}
  Unlike what happens in the case of an homogeneous nonlinearity
  (classically of the form $|u|^{p}u$), replacing $u$ with $\kappa u$
  ($\kappa>0$) 
  in \eqref{eq:logNLS}  has only little effect, since we have
  \begin{equation*}
    i{\partial}_t (\kappa u) +\frac{1}{2} \Delta (\kappa u) =\lambda
    \ln\left(|\kappa u|^2\right)\kappa u - 2\lambda (\ln \kappa )\kappa u \, .
  \end{equation*}
The scaling factor thus corresponds to a purely time-dependent gauge
transform:
\begin{equation*}
  \kappa u (t,x) e^{2it \lambda \ln \kappa}
\end{equation*}
solves \eqref{eq:logNLS} (with initial datum $\kappa u_0$). In
particular, the $L^2$-norm of the initial datum does not influence the
dynamics of the solution. This suggests that the presence of the
  nonlinearity has a 
rather weak influence on the dynamics, since it shares a feature with
linear equations. We will see however that the effects of the
nonlinear term are, from various perspectives, much stronger than the
effects of a defocusing power-like nonlinearity.
Nevertheless, if we denote by $u_\kappa$ the above solution, we readily
compute
\begin{equation*}
  \frac{du_\kappa }{d\kappa}= \left(1+2it\lambda\right)u_1e^{2it \lambda \ln \kappa},\quad 
\frac{d^2u_\kappa }{d\kappa^2}=
\frac{2it\lambda}{\kappa}\left(1+2it\lambda\right)u_1e^{2it \lambda \ln
  \kappa}. 
\end{equation*}
This shows that for any $t\not =0$, the flow map $u_0\mapsto u(t)$
fails to be $C^2$ at the origin, whichever the functional spaces 
considered to measure this regularity. 
\end{remark}
Note that whichever the sign of $\lambda$, the energy $E$ has no
definite sign. The distinction between focusing or defocusing
nonlinearity is thus a priori ambiguous. We shall see however that in
the case $\lambda <0$, no solution  is dispersive, while for~$\lambda
>0$,  solutions have a dispersive behavior (with a non-standard
rate of dispersion). This is why we choose to call \emph{defocusing}
the case $\lambda >0$. 
\subsection{The focusing case}
\label{sec:focusing}

In \cite{CaHa80} (see also \cite{CazCourant}), the Cauchy problem is
studied in the case $\lambda <0$. 
Define
\begin{equation*}
  W:= \Big \{ u\in H^1({\mathbb R}^d) \, , \,  x\mapsto |u(x)|^2\ln |u(x)|^2\in L^1({\mathbb R}^d)\Big\} \, . 
\end{equation*}
\begin{theorem}[Th\'eor\`eme~2.1 from \cite{CaHa80}, see also
  Theorem~9.3.4 from \cite{CazCourant}] Let the initial data~$u_0$
  belong to~$ W$.     In the case when~$\lambda <0$, there exists a
  unique, global  
  solution $u\in C({\mathbb R};W)$ to \eqref{eq:logNLS}. In particular, for all
  $t\in {\mathbb R}$, $|u(t,\cdot)|^2\ln|u(t,\cdot)|^2$ belongs to~$ L^1({\mathbb R}^d)$, and the
  mass~$M(u)$ and the energy $E(u)$ are independent of time. 
\end{theorem}

In the case when~$\lambda<0$, it can be proved that there is no
dispersion for large times.  Indeed the following result  holds.
\begin{lemma}[Lemma~3.3 from \cite{Caz83}]\label{lem:energie}
  Let $\lambda<0$ and $k<\infty$ such that
\[
L_k :=\Big \{u\in W, \ \|u\|_{L^2({\mathbb R}^d)}=1,\ E(u)\le k\Big\}\not
  =\emptyset \, .
\]
 Then
  \begin{equation*}
    \inf_{\substack{u\in L_k \\ 1\le p\le \infty}}\|u\|_{L^p({\mathbb R}^d)}>0 \, . 
  \end{equation*}
\end{lemma}
This lemma, along with the conservation of the energy for
\eqref{eq:logNLS}, indicates that in the case $\lambda<0$, the solution
to \eqref{eq:logNLS} is not dispersive: typically, its $L^\infty$ norm
is bounded from below.
Actually in the case of Gaussian initial data, some solutions are even
known to be periodic in time, as proved in~\cite{AMS14} (and already
noticed in \cite{BiMy76}).
\begin{proposition}[\cite{AMS14}]\label{negativeperiodic}
In the case~$\lambda <0$,   the \emph{Gausson}
\[
 \exp \left(-2i\lambda \omega t
+\omega+d/2+\lambda|x|^2\right)
\]
 is a solution 
to~\eqref{eq:logNLS} for any period~$\omega \in {\mathbb R}$.
 \end{proposition}

We emphasize that several results address the existence
of stationary solutions to~\eqref{eq:logNLS} in the case $\lambda<0$,
and the orbital stability of the Gausson; see
e.g. \cite{BiMy76,Caz83,AMS14,Ar16}. 
We also note that when $\lambda>0$, the above formula yields a
solution to \eqref{eq:logNLS} which is $C^\infty$, time periodic, but
not even a tempered distribution.

\subsection{Main results}
\label{sec:results}
Throughout the rest of this paper, we assume $\lambda >0$.

\subsubsection{The Cauchy problem} For~$0<\alpha\le 1$, we have
\begin{equation*}
{\mathcal F}(H^\alpha):= \Big \{ u\in L^2({\mathbb R}^d) \, , \,  x\mapsto \langle x\rangle
^{\alpha}  u(x) \in L^2({\mathbb R}^d)\Big\} \, , 
\end{equation*}
where~$ \langle x\rangle := \sqrt{1+|x|^2}$ and ${\mathcal F}$ denotes the Fourier
transform (whose normalization is irrelevant here), with norm
\[
\|u\|_{{\mathcal F}(H^\alpha)}:= \|  \langle x\rangle ^{ \alpha}  u(x) \|_{L^2({\mathbb R}^d)} \,.
\]
Note that for any $\alpha>0$, ${\mathcal F}(H^\alpha) \cap H^1\subset W$.
The Cauchy problem for  \eqref{eq:logNLS} is investigated
in~\cite{GLN10}, where in three space dimensions, the
existence of a unique solution in~$L^\infty({\mathbb R};H^1({\mathbb R}^3)) \cap
C({\mathbb R};L^2({\mathbb R}^3)) $ is proved as soon as the initial data belongs
to~${\mathcal F}(H^{1/2}) \cap H^1({\mathbb R}^3)$. 
Actually it is possible to  improve slightly that result into the
following theorem. 
\begin{theorem}\label{positivecauchy}
 Let the initial data~$u_0$ belong to~$ {\mathcal F}(H^\alpha) \cap H^1({\mathbb R}^d)$
 with~$0<\alpha\le 1$. 
In the case when~$\lambda >0$,   there exists a unique, global 
  solution $u\in   L^\infty_{\rm loc} ({\mathbb R};{\mathcal F}(H^\alpha) \cap H^1)$
  to~\eqref{eq:logNLS}. Moreover  the 
  mass~$M(u)$, the angular momentum $J(u)$,  and the energy~$E(u)$ are
  independent of time.  
If in addition $u_0\in H^2({\mathbb R}^d)$, then $u\in L^\infty_{\rm loc}
({\mathbb R};H^2)$. 
\end{theorem}
The main focus of this paper concerns   large time asymptotics of the solution. The
situation is very different 
from the~$\lambda<0$ case described above (see
Proposition~\ref{negativeperiodic}). Indeed  we can prove that (some)
solutions tend to zero in~$L^\infty$ for large time, while the~$H^s$
norm is always unbounded for $s>0$. 

As often in the context of nonlinear Schr\"odinger equations, we denote
by
\begin{equation*}
  \Sigma=H^1\cap {\mathcal F}(H^1)({\mathbb R}^d)= \{f\in H^1({\mathbb R}^d),\ x\mapsto |x|f(x)\in
  L^2({\mathbb R}^d)\}. 
\end{equation*}

\subsubsection{Long time behavior} 
We show that three new features  characterize the dynamics associated
to \eqref{eq:logNLS}:
\begin{itemize}
\item The standard dispersion of the Schr\"odinger equation, in
  $t^{-d/2}$, is altered by a logarithmic factor, in $\left(t\sqrt{\ln
    t}\right)^{-d/2}$. This acceleration is of course an effect of the
  nonlinearity.
\item All the positive Sobolev norms of the solution grow
  logarithmically in time (explicit rate), due to some discrepancy between the
  dispersive rate and a universal quadratic oscillation in space.
\item Up to a rescaling, the modulus of the solution converges for
  large time to a \emph{universal} Gaussian profile.
\end{itemize}
Before stating the general result, let us introduce the universal
dispersion rate~$\tau$ through the following lemma.
We define from now on the function
\begin{equation}\label{eq:defell}
\ell (t):= \frac{\ln \ln t}{ \ln t} \, \cdotp
\end{equation}

\begin{lemma}[Universal dispersion]\label{lem:tau}
 Consider the ordinary differential equation
\begin{equation}\label{eq:tau}
  \ddot \tau = \frac{2\lambda }{\tau} \, ,\quad \tau(0)=1\, ,\quad \dot
  \tau(0)=0\, .
\end{equation}
It has a unique solution $\tau\in C^2(0,\infty)$, and it satisfies, as
$t\to \infty$, 
\begin{equation*}
  \tau(t)= 2t \sqrt{\lambda \ln t}\Big(1+{\mathcal O}\big(\ell(t)\big) \Big) \, ,\quad \dot
  \tau(t)=2\sqrt{\lambda\ln  t}\Big(1+{\mathcal O}\big(\ell(t)\big)\Big)\, .
\end{equation*} 
\end{lemma}
We now state the main result of this paper. Denote by
\[
\gamma(x):=e^{-|x|^2/2}.
\]
 \begin{theorem}\label{theo:main}
Let $u_0\in \Sigma$,  and rescale the solution
provided by 
Theorem~{\rm\ref{positivecauchy}} to $v=v(t,y)$ by setting
\begin{equation}
  \label{eq:uv}
  u(t,x)
  =\frac{1}{\tau(t)^{d/2}}v\left(t,\frac{x}{\tau(t)}\right)
\frac{\|u_0\|_{L^2({\mathbb R}^d)}}{\|\gamma\|_{L^2({\mathbb R}^d)}} 
\exp \Big({i\frac{\dot\tau(t)}{\tau(t)}\frac{|x|^2}{2}} \Big) \, . 
\end{equation}
There exists $C$ such that for
all $t\ge 0$,
\begin{equation}\label{eq:apv}
 \int_{{\mathbb R}^d}\left(1+|y|^2+\left|\ln
    |v(t,y)|^2\right|\right)|v(t,y)|^2dy +\frac{1}{\tau(t)^2}\|\nabla_y
  v(t)\|^2_{L^2({\mathbb R}^d)}\le C \, .  
\end{equation}
 We have
moreover
\begin{equation}\label{eq:moments}
   \int_{{\mathbb R}^d}
  \begin{pmatrix}
    1\\
y\\
|y|^2
  \end{pmatrix}
|v(t,y)|^2dy\Tend t \infty 
 \int_{{\mathbb R}^d}
  \begin{pmatrix}
    1\\
y\\
|y|^2
  \end{pmatrix}
\gamma^2(y)dy \, .
\end{equation}
Finally, 
\begin{equation}\label{eq:weaklimitv}
  |v(t,\cdot)|^2 \mathop{\rightharpoonup}\limits_{t\to \infty}
  \gamma^2 
\quad  \text{weakly in }L^1({\mathbb R}^d) \, . 
\end{equation}
\end{theorem}

\begin{remark}
The scaling factor in \eqref{eq:uv} is here to normalize the
$L^2$-norm of $v$ to be the same as the $L^2$-norm of $\gamma$, and
is of no influence regarding the statement, in view of
Remark~\ref{rem:scaling}. 
To the best of our knowledge, this is the first time that  a universal
profile is observed  
 for the large time behavior of solutions to a dispersive, Hamiltonian equation. This profile is reached in a weak sense only in Theorem~\ref{theo:main}, as far as the convergence is concerned, but also because  the modulus of the solution only is captured. This   indicates that a lot of information remains encoded in the oscillations of the solution.  See Section~\ref{sec:comments} for more on the large time asymptotics.
 \end{remark}
\begin{remark}\label{rem:wasserstein}
  As a straightforward consequence, we infer the slightly weaker
  property that $|v(t,\cdot)|^2$ converges to $\gamma^2$ in
  Wasserstein distance:
  \begin{equation*}
    W_2\left(\frac{|v(t,\cdot)|^2}{\pi^{d/2}},\frac{\gamma^2}{\pi^{d/2}}\right)\Tend
    t \infty 0,
  \end{equation*}
where we recall that the Wasserstein distance is defined, for $\nu_1$
and $\nu_2$ probability measures, by 
\begin{equation*}
  W_p(\nu_1,\nu_2)=\inf \left\{ \left(\int_{{\mathbb R}^d\times
    {\mathbb R}^d}|x-y|^pd\mu(x,y)\right)^{1/p};\quad (\pi_j)_\sharp \mu=\nu_j\right\},
\end{equation*}
where $\mu$ varies among all probability measures on ${\mathbb R}^d\times
{\mathbb R}^d$, and $\pi_j:{\mathbb R}^d\times {\mathbb R}^d\to {\mathbb R}^d$ denotes the canonical
projection onto the $j$-th factor (see e.g. \cite{Vi03}). 
\end{remark} 
In the context of nonlinear  Hamiltonian partial differential
equations, a general question is the evolution of Sobolev norms, as
emphasized in \cite{Bo96}. In \cite{HPTV15}, the
  existence of unbounded, in $H^s$ with $s>1$, solutions to the cubic  defocusing 
  Schr\"odinger equation was established for the
  first time, in the case where the equation is considered on the
  domain ${\mathbb R}\times {\mathbb T}^d$, $d\ge 2$: for \emph{some} solutions, the
  Sobolev norms grow logarithmically along a sequence of times.
 See also \cite{Iturbulent,GuKa15,GHP16}, in the space periodic
 case. For other equations (cubic 
 Szeg\H{o} equation or  
half-wave equation),
with \emph{specific 
initial data}, a growth rate can be exhibited, possibly along a
sequence of time; see
\cite{GeGr12,GeGr15,Po11}. In particular, for the cubic Szeg\H{o}
equation, the generic (in the sense of Baire) growth of Sobolev norms
with superpolynomial rates is established in \cite{GeGr17}. Combining
the approaches of \cite{HPTV15} and \cite{GeGr17}, it was proved in
\cite{Xu17} that a system of half-wave--Schr\"odinger on the cylinder
${\mathbb R}_x\times {\mathbb T}_y$  possesses unbounded solutions in $L^2_xH^s_y$ for $s>1/2$,
where the growth rate is a superpolynomial function of $\ln t$.  
We show that in the case of
\eqref{eq:logNLS}, the Sobolev norms of \emph{all} solutions grow, regardless
of the data, and we
give a sharp rate.
\begin{corollary}\label{cor:Hs}
   Let $u_0\in \Sigma\setminus\{0\}$, and $0<s\le 1$. The solution
  to \eqref{eq:logNLS} satisfies, as $t\to \infty$,
\begin{equation*}
  \|\nabla u(t)\|_{L^2({\mathbb R}^d)}^2 \Eq t \infty 2\lambda d \|u_0\|_{L^2({\mathbb R}^d)}^2 
\ln  t ,
\end{equation*}
and
  \begin{equation*}
 \left(\ln t\right)^{s/2}\lesssim   \|u(t)\|_{\dot H^s({\mathbb R}^d)}\lesssim \left(\ln t\right)^{s/2},
  \end{equation*}
where $\dot H^s({\mathbb R}^d)$ denotes the standard homogeneous Sobolev space.
\end{corollary}
 The weak convergence in the last point of Theorem~\ref{theo:main}
 may seem puzzling. In the case of Gaussian
 initial data, we can prove that the convergence is strong. 
\begin{corollary}[Strong convergence in the Gaussian
  case]\label{cor:entropie-gauss} 
  Suppose that the initial data~$u_0$ is a Gaussian,
\begin{equation*}
 u_0(x) =  b_0 
  \exp \Big( -\frac{1}{2}\sum_{j=1}^d a_{0j} (x_j-x_{0j})^2\Big)
\end{equation*}
with $b_0,a_{0j}\in {\mathbb C}$, $\RE a_{0j}>0$, and $x_{0j}\in
{\mathbb R}$. Then, with  $v$ 
  given by \eqref{eq:uv}, the relative
      entropy of $|v|^2$  goes to zero for large time:
      \begin{equation*}
        \int_{{\mathbb R}^d}|v(t,y)|^2\ln\left|\frac{v(t,y)}{\gamma(y)}\right|^2dy\Tend
        t \infty 0 \, ,
      \end{equation*}
and the convergence of~$|v|^2$ to~$\gamma^2$ is strong in $L^1$:
\begin{equation*}
  \left\||v(t,\cdot)|^2-\gamma^2\right\|_{L^1({\mathbb R}^d)}\Tend t \infty 0 \, .
\end{equation*}
\end{corollary}

\subsection{Comments and further result}
\label{sec:comments}

In the linear case
\begin{equation}
  \label{eq:linear}
  i{\partial}_t u_{\rm free}+\frac{1}{2}\Delta u_{\rm free} =0 \, , \quad
  u_{{\rm free}\mid t=0}=u_0 \, , 
\end{equation}
the integral representation
\begin{equation*}
  u_{\rm free}(t,x) = \frac{1}{(2i\pi t)^{d/2}}\int_{{\mathbb R}^d}e^{i\frac{|x-y|^2}{2t}}u_0(y)dy
\end{equation*}
readily yields the well-known dispersive estimate
\begin{equation*}
  \|u_{\rm free}(t)\|_{L^\infty({\mathbb R}^d)}\lesssim \frac{1}{|t|^{d/2}}\|u_0\|_{L^1({\mathbb R}^d)} \, .
\end{equation*}
Moreover, defining the Fourier transform as
\begin{equation*}
  {\mathcal F}
f(\xi):=\hat f(\xi):=\frac{1}{(2\pi)^{d/2}}\int_{{\mathbb R}^d}f(x)e^{-ix\cdot \xi} dx \, , 
\end{equation*}
we have the standard asymptotics (see e.g. \cite{Rauch91}),
\begin{equation}\label{eq:asymfree}
   \left\lVert u_{\rm free}(t) - A(t)u_0\right\rVert_{L^2({\mathbb R}^d)} \Tend t
    {\pm\infty} 0, \quad A(t)u_0(x) := \frac{1}{(it)^{d/2}}\hat
  u_0\left(\frac{x}{t}\right)e^{i\frac{\lvert x\rvert^2}{2t}} \, ,
\end{equation}
a formula which has proven very useful in the nonlinear (long range)
scattering theory (see e.g. \cite{Ginibre,HaNa06}).

 In the case of the
defocusing nonlinear Schr\"odinger equation with power-like
nonlinearity, 
\begin{equation}\label{eq:NLS}
  i{\partial}_t u+\frac{1}{2}\Delta u = |u|^{2\sigma}u \, ,\quad u_{\mid t=0}=u_0 \, ,
\end{equation}
if $\sigma$ is sufficiently large (say $\sigma>2/d$ if $u_0\in H^1({\mathbb R}^d)$,
even though this bound can be lowered if in addition $\hat u_0\in H^1({\mathbb R}^d)$),
then there exists $u_+\in H^1({\mathbb R}^d)$ such that in $L^2({\mathbb R}^d)$,
\begin{equation*}
  u(t,x)\Eq t\infty e^{i\frac{t}{2}\Delta}u_+(x)\Eq t \infty \frac{1}{(it)^{d/2}}\hat
  u_+\left(\frac{x}{t}\right)e^{i\frac{\lvert x\rvert^2}{2t}} \, ,
\end{equation*}
where the last relation stems from \eqref{eq:asymfree}.
Therefore,
Theorem~\ref{theo:main} shows that unlike in the free case
\eqref{eq:linear} or in the above nonlinear case~\eqref{eq:NLS}, the dispersion is
modified (it is even enhanced the larger the $\lambda$), and the
asymptotic profile $\hat u_+$ (with $u_+=u_0$ in the 
free case), which depends on the initial profile, is replaced by a
universal one (up to a normalizing factor),
\begin{equation*}
  \frac{\|u_0\|_{L^2}}{\pi^{d/4}} e^{-|x|^2/2} \, .
\end{equation*}
As already pointed out, the nonlinearity is responsible for the new
dispersive rate, as it introduces a logarithmic factor. In particular,
no scattering result relating the dynamics of \eqref{eq:logNLS} to the
free dynamics $e^{i\frac{t}{2}\Delta}$ must be expected. This
situation can be compared with the more familiar one with low power
nonlinearity, where a long range scattering theory is (sometimes)
available. If $\sigma\le 1/d$ in \eqref{eq:NLS}, then $u$ cannot be be
compared with a free evolution for large time, in the sense that if
for some $u_+\in L^2({\mathbb R}^d)$,
\begin{equation*}
  \|u(t)-e^{i\frac{t}{2}\Delta}u_+\|_{L^2({\mathbb R}^d)}\Tend t \infty 0 \, ,
\end{equation*}
then $u=u_+=0$ (\cite{Barab}). In the case $\sigma=1/d$, $d=1,2,3$,  a
nonlinear phase 
modification of $e^{i\frac{t}{2}\Delta}$ must be incorporated in order
to describe the asymptotic behavior of $u$
(\cite{Ozawa91,HaNa06}). The same is true when \eqref{eq:NLS} is
replaced with the Hartree equation \cite{GV01,Na02a,Na02b}. In all these
cases, as well as for some
quadratic nonlinearities in dimension~3 (\cite{HMN03}), the dispersive rate of
the solution remains the same as in the free case, of order
$t^{-d/2}$. Note however  that a similar logarithmic perturbation of
the dispersive rate was observed in \cite{HaNa15}, for the equation
\begin{equation}\label{eq:HaNa15}
  i{\partial}_tu+\frac{1}{2}{\partial}_x^2 u =i\lambda u^3 +|u|^2u \, ,\quad x\in {\mathbb R} \, ,
\end{equation}
with $\lambda\in {\mathbb R}$, $0<|\lambda|<\sqrt 3$. More precisely, the
authors construct   small
solutions satisfying the bounds
\begin{equation*}
  \frac{1}{\sqrt t (\ln t)^{1/4}}\lesssim \sup_{|x|\le \sqrt
    t}|u(t,x)|\lesssim  \frac{1}{\sqrt t (\ln t)^{1/4}} \, ,\quad \text{as
  }t\to \infty \, . 
\end{equation*}
An important difference with \eqref{eq:logNLS} though is that the
$L^2$-norm of the solution of~\eqref{eq:HaNa15} is not preserved by
the flow, and that \eqref{eq:HaNa15} has no Hamiltonian structure. 
\smallbreak

On the other hand, \eqref{eq:uv} brings out a universal spatial
oscillation, namely the term $\exp\left(i\frac{\dot
  \tau(t)}{\tau(t)}\frac{|x|^2}{2}\right)$. In view of Lemma~\ref{lem:tau},
we have 
\begin{equation*}
  \frac{\dot
  \tau(t)}{\tau(t)}
-\frac{1}{t}=\frac{{\mathcal O}\left(\ell(t)\right)}{t\left(1+{\mathcal O}\left(\ell(t)\right)\right)},\quad
\text{hence }\exp\left(i\frac{\dot
  \tau(t)}{\tau(t)}\frac{|x|^2}{2}\right)\Eq t \infty e^{i\frac{|x|^2}{2t}},
\end{equation*}
the same universal oscillation as for the linear Schr\"odinger
equation \eqref{eq:linear}, see \eqref{eq:asymfree}.
\smallbreak

Finally let us remark that the convergence to a universal profile is
reminiscent of what happens for the linear heat equation on ${\mathbb R}^d$. 
Indeed, let $u$ solve
\begin{equation}
  \label{eq:heat}
  {\partial}_t u = \frac{1}{2}\Delta u,\quad (t,x)\in {\mathbb R}_+\times {\mathbb R}^d,\quad
  u_{\mid t=0}=u_0.
\end{equation}
This equation can be solved thanks to Fourier analysis,
$  \hat u(t,\xi)= \hat u_0(\xi)e^{-\frac{t}{2}|\xi|^2}$. This allows
to show the following standard asymptotics (see e.g. \cite{Rauch91}): 
\begin{equation*}
  u(t,x)\Eq t\infty \frac{m}{(2\pi t)^{d/2}}  e^{-|x|^2/(2t)}\quad
  \text{in }L^2\cap L^\infty({\mathbb R}^d),\text{ provided }
  m:=\int_{{\mathbb R}^d}u_0(x)dx\not = 0.
\end{equation*}
We note that at
leading order, the only role played by the initial data is the presence
of the total mass $m$. The asymptotic profile is \emph{universal}, and
corresponds to the Gaussian $\gamma$. 
\smallbreak

In a nonlinear setting, our result is reminiscent of the
works~\cite{GaWa02,GaWa05,GaWa05bis} on the 
Navier-Stokes equations: there it is proved that up to a rescaling
which corresponds to the natural scaling of the equations, the
vorticity converges strongly to a Gaussian which is known as the Oseen
vortex. The main argument, as in the present case, is the reduction to
a Fokker-Planck equation. 
\smallbreak

In view of the above discussion, one may ask if the phenomena stated
in Theorem~\ref{theo:main} and Corollary~\ref{cor:Hs} are bound to the
very special structure of the nonlinearity. Our final result shows
that it is not the case, inasmuch as these results remain (possibly up
to a uniqueness issue) when an energy-subcritical defocusing
power-like nonlinearity is added,
\begin{equation}
  \label{eq:power}
  i{\partial}_t u + \frac{1}{2}\Delta u = \lambda u \ln(|u|^2)+
  \mu|u|^{2\sigma}u,\quad u_{\mid t=0}=u_0.
\end{equation}
The mass and angular momentum of $u$ are the same as before, and they
are formally conserved, as well as the energy
\begin{equation*}
  E(u(t))= \frac{1}{2}\|\nabla
  u(t)\|_{L^2({\mathbb R}^d)}^2+\lambda\int_{{\mathbb R}^d} |u(t,x)|^2\ln |u(t,x)|^2dx
  +\frac{\mu}{\sigma+1}\int_{{\mathbb R}^d}|u(t,x)|^{2\sigma+2}dx. 
\end{equation*}
\begin{theorem}\label{theo:power}
  Let $\lambda,\mu>0$ and $0<\sigma<2/(d-2)_+$. For $u_0\in \Sigma$,
  \eqref{eq:power} has a solution $u\in L^\infty_{\rm
    loc}({\mathbb R};\Sigma)$. It is unique (at least) if $d=1$. Its mass,
  angular momentum and energy are independent of time. Setting
\begin{equation*}
 u(t,x)
  =\frac{1}{\tau(t)^{d/2}}v\left(t,\frac{x}{\tau(t)}\right)
\frac{\|u_0\|_{L^2({\mathbb R}^d)}}{\|\gamma\|_{L^2({\mathbb R}^d)}} 
\exp \Big({i\frac{\dot\tau(t)}{\tau(t)}\frac{|x|^2}{2}} \Big) \, ,
\end{equation*}
we have
\begin{equation*}
   \int_{{\mathbb R}^d}
  \begin{pmatrix}
    1\\
y\\
|y|^2
  \end{pmatrix}
|v(t,y)|^2dy\Tend t \infty 
 \int_{{\mathbb R}^d}
  \begin{pmatrix}
    1\\
y\\
|y|^2
  \end{pmatrix}
\gamma^2(y)dy \, ,
\end{equation*}
and
\begin{equation*}
  |v(t,\cdot)|^2 \mathop{\rightharpoonup}\limits_{t\to \infty}
  \gamma^2 
\quad  \text{weakly in }L^1({\mathbb R}^d) \, . 
\end{equation*}
Finally, for 
$0<s\le 1$, $u$ satisfies, as $t\to \infty$,
\begin{equation*}
  \|\nabla u(t)\|_{L^2({\mathbb R}^d)}^2 \Eq t \infty 2\lambda d \|u_0\|_{L^2({\mathbb R}^d)}^2 
\ln  t ,
\end{equation*}
and
  \begin{equation*}
 \left(\ln t\right)^{s/2}\lesssim   \|u(t)\|_{\dot H^s({\mathbb R}^d)}\lesssim \left(\ln
 t\right)^{s/2}. 
  \end{equation*}
\end{theorem}
The above result shows that the logarithmic nonlinearity, far from being a weak
nonlinearity, dictates the dynamics even in the presence of an
energy-subcritical defocusing 
power-like nonlinearity, in sharp contrast with the scattering results
of the case $\lambda=0$ mentioned above. The dynamics due to the logarithmic
nonlinearity is thus rather stable.

\subsection{Outline of the paper}

In Section~\ref{sec:cauchy}, we give the proof of
Theorem~\ref{positivecauchy}. Section~\ref{sec:gaussian} contains
explicit computations in the Gaussian case, paving
  the way for the general case. The main step consists in a reduction to ordinary differential equations, and  the computations concerning the main ODEs
 are gathered in this section, including the
proof of Lemma~\ref{lem:tau}. The first part of
Theorem~\ref{theo:main}, that is everything  except~\eqref{eq:weaklimitv}, is
proved in Section~\ref{sec:preparation} and relies on energy estimates with appropriate weights in space. The weak limit
\eqref{eq:weaklimitv} is proved in Section~\ref{sec:end}.  
The main idea consists in using a Madelung transform, which leads to the study of a hyperbolic system which is a variant of the isothermal, compressible Euler equation. A rescaling in the time variable reduces the study to a non autonomous perturbation of the Fokker-Planck equation, and the a priori estimates obtained in  Section~\ref{sec:preparation} imply that a weak limit of the solution satisfies the Fokker-Planck equation for large times. Since it is known that the large time behavior of the solution to the  Fokker-Planck equation is the centered Gaussian, a tightness argument on the rescaled solution concludes the proof.
The proofs of the corollaries are given in 
Section~\ref{sec:entropie}, and the main arguments leading to
Theorem~\ref{theo:power}  are given in Section~\ref{sec:power}.

\subsection*{Acknowledgements} The authors wish to thank Kleber
Carrapatoso, Laurent Desvillettes, Erwan Faou, Matthieu Hillairet and
C\'edric Villani for enlightening discussions, and the referees for 
their careful reading of the paper and their numerous constructive comments.

\section{Cauchy problem: proof of 
Theorem~\ref{positivecauchy}} 
\label{sec:cauchy}
In this section we   sketch the proof of the existence of a unique
weak solution, which follows very standard ideas (see~\cite{CazCourant,CaHa80}).  We first prove, in Section~\ref{sec:existence}, the existence of weak solutions by solving an approximate system and passing to the limit in the approximation parameter. This produces a weak solution, whose uniqueness is proved in Section~\ref{sec:uniqueness}. The proof of the  propagation of higher regularity is given in Section~\ref{sec:higher}.

\subsection{Existence}\label{sec:existence}
To prove the existence of a weak solution we proceed by approximating the equation as follows: consider for all~$\varepsilon \in (0,1)$ the equation
\begin{equation}
  \label{eq:logNLSeps}
  i{\partial}_t u_\varepsilon +\frac{1}{2} \Delta u_\varepsilon =\lambda
  \ln\left(\varepsilon+|u_\varepsilon|^2\right)u_\varepsilon\, ,\quad u_{\varepsilon\mid t=0} =u_{0} \, . 
\end{equation}
 Equation~\eqref{eq:logNLSeps} is easily solved in~$C({\mathbb R};L^2({\mathbb R}^d))$ since it is subcritical in~$L^2$ (see~\cite{CazCourant}). It remains therefore to prove uniform bounds for~$u_\varepsilon(t)$ in~$ {\mathcal F}(H^\alpha)\cap H^1({\mathbb R}^d)$, which will provide compactness in space for the sequence~$u_\varepsilon$. Since time compactness (in~$H^{-2}({\mathbb R}^d)$) is a direct consequence of the equation, the Ascoli theorem will then give the result.
Actually once a bound in~$  L^\infty_{\rm loc}({\mathbb R};H^1({\mathbb R}^d))$  is
derived, then the~$L^\infty_{\rm loc}({\mathbb R};{\mathcal F}(H^\alpha) )$ bound can be
obtained directly thanks to the following computation: define 
\[
I_{\varepsilon,\alpha} (t):= \int_{{\mathbb R}^d} \langle x \rangle^{2\alpha} |u_{\varepsilon}|^2(t,x) \, dx \, .
\]
Then multiplying the equation by~$ \langle  x \rangle^{2\alpha}
\overline u_{\varepsilon}$ and integrating in space provides   
\[
\begin{aligned}
\frac d{dt} I_{\varepsilon,\alpha} (t) &= 2\alpha\IM \int  \frac{x
  \cdot \nabla  u_{\varepsilon}  }{  \langle x \rangle^{2-2\alpha}} \,
\overline u_{\varepsilon}(t)  \, dx \\ 
& \le 2\alpha\|\left\langle x\right\rangle^{2\alpha-1}u_\varepsilon(t)
\|_{L^2({\mathbb R}^d)} \|\nabla 
u_\varepsilon(t)  \|_{L^2({\mathbb R}^d)} \\
&\le 2\alpha\|\left\langle x\right\rangle^{\alpha}u_\varepsilon(t)  \|_{L^2({\mathbb R}^d)} \|\nabla
u_\varepsilon(t)  \|_{L^2({\mathbb R}^d)}
\, ,
\end{aligned}
\]
where the last estimate stems from the property $\alpha\le 1$. Therefore,
\[
\|u_{\varepsilon}(t)\|_{{\mathcal F}(H^\alpha)}^2 \le \|u_{0}\|_{{\mathcal F}(H^\alpha)}^2 + 2\alpha
\int_0^t \|u_{\varepsilon}(t')\|_{{\mathcal F}(H^\alpha)}   \|\nabla u_\varepsilon(t')\|_{L^2({\mathbb R}^d)} \, dt' \, .
\]
So it remains to compute the~$H^1({\mathbb R}^d)$ norm of~$u_{\varepsilon}(t)$. This
is quite easy since the problem becomes linear in~$\nabla
u_\varepsilon$. Indeed for any~$1 \le j \le d$ one has 
\begin{equation}\label{eq:djueps}
  i{\partial}_t {\partial}_j u_\varepsilon +\frac{1}{2} \Delta {\partial}_ju_\varepsilon =\lambda \ln\left(\varepsilon+|u_\varepsilon|^2\right){\partial}_ju_\varepsilon + 2\lambda \frac{1}{\varepsilon+|u_\varepsilon|^2}  \RE ( {\bar u}_\varepsilon
{\partial}_ju_\varepsilon) u_\varepsilon,
\end{equation}
and we note that~$\displaystyle \Big | \frac{1}{\varepsilon+|u_\varepsilon|^2} 2 \RE ( {\bar u}_\varepsilon
{\partial}_ju_\varepsilon) u_\varepsilon\Big| \le 2 |{\partial}_ju_\varepsilon|  $. We therefore conclude
that~${u_\varepsilon}$ belongs to~$L^\infty_{\rm loc}({\mathbb R};H^1({\mathbb R}^d))$,
uniformly in $\varepsilon$.

\smallbreak

Passing to the limit to obtain a solution conserving mass, angular momentum, and energy is established
in the same way as 
in \cite{CaHa80} (see also~\cite{CazCourant}).  The
existence part of 
Theorem~\ref{positivecauchy} follows.   

\medskip

 For the proof of uniqueness below it is   useful to prove that any
 solution in the class $L^\infty_{\rm loc}(\mathbb R;\mathcal F(H^\alpha)\cap H^1)$
 belongs actually to~$C({\mathbb
   R};(L^2 \cap H^1_{\rm weak})({\mathbb R}^d))$, which allows to make
 sense of the initial data in~$H^1 ({\mathbb R}^d)$. The method of
 proof follows the idea of~\cite{CaHa80}, and is in fact easier due to
 our functional setting. Let~$u$ be such a solution, then
 clearly~$\Delta u$ belongs to~$L^\infty_{\rm loc}({\mathbb R};H^{-1}
 ({\mathbb R}^d))$ and we claim that~$ u \ln |u|^2 $ belongs
 to~$L^\infty_{\rm loc}({\mathbb R};L^2  ({\mathbb R}^d))$. Indeed
 there holds 
 \[
 \int |u|^2 (\ln |u|^2)^2  \lesssim  \int |u|^{2-\epsilon}  
+  \int |u|^{2+\epsilon}   
\]
 for all~$\epsilon>0$, and moreover we have the estimate
 \begin{equation}\label{interpolationinequality}
  \int_{{\mathbb R}^d} |u|^{2-\epsilon} \lesssim \|u\|_{L^2}^{2- \epsilon - \frac{d\epsilon}{2\alpha}}
  \|x^\alpha u\|_{L^2}^{ \frac{d\epsilon}{2\alpha}} \, ,
\end{equation}
for~$0<\epsilon<\frac{4\alpha}{d+2\alpha}$, which can be
readily proved by an interpolation method (cutting the integral
into~$|y| < R$ and~$|y| > R$, using H\"older inequality and optimizing
over~$R$; see e.g. \cite{CM}). 
This, along with Sobolev embeddings, implies that
\[
 \int |u|^2 (\ln |u|^2)^2  \lesssim   \|u\|_{L^2}^{2- \epsilon - \frac{d\epsilon}{2\alpha}}
  \|x^\alpha u\|_{L^2}^{ \frac{d\epsilon}{2\alpha}} +  \|u\|_{H^ 1}^{2+\epsilon} 
\]
so finally~$\partial_t u$ belongs to~$L^\infty_{\rm loc}({\mathbb R};H^{-1}  ({\mathbb R}^d))$ and the result follows.
\subsection{Uniqueness}\label{sec:uniqueness}
The uniqueness of the solution  constructed above is a consequence of the following lemma.
\begin{lemma}[Lemma 9.3.5 from \cite{CazCourant}]\label{lemunique}
  We have 
  \begin{equation*}
    \left| \IM\left(\left(z_2\ln|z_2|^2
      -z_1\ln|z_1|^2\right)\left(\bar z_2-\bar z_1\right)\right)\right|\le
    4|z_2-z_1|^2 \, ,\quad \forall 
    z_1,z_2\in {\mathbb C} \, .
  \end{equation*}
\end{lemma}
Consider indeed~$u_1$ and~$u_2$  two solutions of~(\ref{eq:logNLS}) as constructed in the previous section. Then the function~$u:=u_1-u_2$ satisfies
\[
 i{\partial}_t u +\frac{1}{2} \Delta u =\lambda \big( \ln\left(|u_1|^2\right)u_1- \ln\left(|u_2|^2\right)u_2\big )
\]
and the regularity of~$u_1$ and~$u_2$ enables one to write an energy estimate in~$L^2$ on this equation. We get directly
 \begin{equation}\label{eq:lipL2}
\begin{aligned}
\frac12 \frac d {dt} \|u(t) \|_{L^2({\mathbb R}^d)}^2& = \lambda  \IM \int_{{\mathbb R}^d} \big( \ln\left(|u_1|^2\right)u_1- \ln\left(|u_2|^2\right)u_2\big ) (\bar u_1-\bar u_2) (t) \, dx
\\
& \le 4 \lambda \|u(t) \|_{L^2({\mathbb R}^d)}^2 
\end{aligned}
\end{equation}
thanks to Lemma~\ref{lemunique}. 
 Uniqueness (and in fact stability in~$L^2$) follows directly, by integration in time.

\subsection{Higher regularity}
\label{sec:higher}

As in \cite{CaHa80}, the idea is to consider time derivatives. This
fairly general idea in the context of nonlinear Schr\"odinger
equations (see \cite{CazCourant}) is all the more precious in the
present framework that the logarithmic nonlinearity is very little
regular. In particular, we emphasize that if $u_0\in H^k({\mathbb R}^d)$, $k\ge
3$, we cannot guarantee in general that this higher regularity is
propagated.
\smallbreak

To complete the proof of Theorem~\ref{positivecauchy}, assume that
$u_0\in {\mathcal F}(H^\alpha)\cap H^2$, for some~$\alpha>0$. We already know that
a unique, global, weak solution $u\in L^\infty_{\rm
  loc}({\mathbb R};{\mathcal F}(H^\alpha)\cap H^1)$ is obtained by the procedure described in
the previous subsection, that is, as the limit of $u_\varepsilon$ solution to
\eqref{eq:logNLSeps}. The idea is that for all $T>0$, there exists $C=C(T)$
independent of $\varepsilon\in (0,1)$ such that
\begin{equation*}
  \sup_{-T\le t\le T}\|{\partial}_t u_\varepsilon(t)\|_{L^2({\mathbb R}^d)}\le C \, .
\end{equation*}
Indeed, we know directly from \eqref{eq:logNLSeps} that 
\begin{equation*}
  {\partial}_t u _{\varepsilon\mid t=0} = \frac{i}{2}\Delta u_0-i\lambda
  \ln\left(\varepsilon+|u_0|^2\right)u_0 \in L^2({\mathbb R}^d) \, ,
\end{equation*}
uniformly in $\varepsilon$, in view of the pointwise estimate
\begin{equation*}
  \left| \ln\left(\varepsilon+|u_0|^2\right)u_0\right|\le C
  \left(|u_0|^{1+\eta}+|u_0|^{1-\eta}\right) \, , 
\end{equation*}
where $\eta>0$ can be chosen arbitrarily small, and $C$ is independent
of $\varepsilon\in (0,1)$. Then we can replace the spatial derivative ${\partial}_j$ in~(\ref{eq:djueps})   with the time derivative~${\partial}_t$, and infer that
${\partial}_t u_\varepsilon\in L^\infty_{\rm loc}({\mathbb R};L^2({\mathbb R}^d))$, uniformly in
$\varepsilon$: by passing to the limit (up
to a subsequence), ${\partial}_t u \in  L^\infty_{\rm loc}({\mathbb R};L^2({\mathbb R}^d))$. Using the
equation \eqref{eq:logNLS}, we infer that $\Delta u \in
L^\infty_{\rm loc}({\mathbb R};L^2({\mathbb R}^d))$.  This concludes the proof of Theorem~\ref{positivecauchy}. \qed

\section{Propagation of Gaussian data} 
\label{sec:gaussian}

As noticed already in
\cite{BiMy76}, an important feature of \eqref{eq:logNLS} is that the evolution of initial Gaussian data remains
Gaussian. Since \eqref{eq:logNLS} is invariant by translation in
space, we may consider centered Gaussian initial data. The following
result is a crucial guide for the general case. 
\begin{theorem}\label{theo:gaussian}
Let~$\lambda >0$,  and consider the initial data
\begin{equation}
\label{eq:gauss-init}
  u_0(x) =  b_0 
  \exp \Big( -\frac{1}{2}\sum_{j=1}^d a_{0j} x_j^2\Big)
\end{equation}
with $b_0,a_{0j}\in {\mathbb C}$, $\alpha_{0j}=\RE a_{0j}>0$. Then the  
solution~$u$ to \eqref{eq:logNLS} is given by
\begin{equation*}
  u(t,x) = b_0\prod_{j=1}^d\frac{1}{\sqrt{r_j(t)}} 
\exp \Big(i\phi_j(t)-\alpha_{0j} \frac{x_j^2}{2r_j^2(t)}+i\frac{\dot r_j(t) }{r_j(t)}\frac{x_j^2}{2}\Big)
\end{equation*}
for some real-valued functions $\phi_j,r_j$ depending on time only,
such that, as $t\to \infty$,
\begin{equation}
\label{eq:asymprj}
  r_j(t)= 2t \sqrt{\lambda \alpha_{0j} \ln t} \, \Big(1+{\mathcal O}\big(\ell(t)\big)\Big) \, ,\quad \dot
  r_j(t)=2\sqrt{\lambda \alpha_{0j}\ln  t} \, \Big(1+{\mathcal O}\big(\ell(t)\big)\Big) \, ,
\end{equation}
where $\ell$ is defined in \eqref{eq:defell}.
In particular, as~$t \to \infty$,
\[
\|u(t)\|_{L^\infty({\mathbb R}^d)} \sim \frac{1}{\left(t\sqrt{\ln t}\right)^{d/2}}\frac{\|u_0\|_{L^2}}{
  \left(2\lambda\sqrt{2\pi}\right)^{d/2}} \, \cdotp 
\]
On the other hand~$u$ belongs to~$L^\infty_{\rm loc}({\mathbb R};H^1({\mathbb R}^d))$
and as~$t \to \infty$ 
\[
\|\nabla u(t)\|_{L^2({\mathbb R}^d)}^2 \Eq t \infty 2\lambda d
\|u_0\|_{L^2({\mathbb R}^d)}^2  \ln  t  .\]
 \end{theorem}

\subsection{From \eqref{eq:logNLS} to ordinary differential equations}
\subsubsection{The Gaussian structure}
As noticed in \cite{BiMy76}, the flow of \eqref{eq:logNLS} preserves any
initial Gaussian structure. We consider the data given by~(\ref{eq:gauss-init}), and we seek the solution
$u$ to \eqref{eq:logNLS} under the form
\begin{equation}\label{eq:gauss}
  u(t,x) = b(t)\exp \Big( {-\frac{1}{2}\sum_{j=1}^d a_j(t) x_j^2}\Big) \, ,
\end{equation}
with $\RE a_j(t)>0$. With $u$ of this form, \eqref{eq:logNLS} becomes
equivalent to 
\begin{equation*}
  i{\partial}_t u +\frac{1}{2}\Delta u = \lambda\Big (\ln |b(t)|^2 -\sum_{j=1}^d
  \RE a_j(t) x_j^2\Big)u \, , \quad u_{\mid t=0}=u_0 \, . 
\end{equation*}
This is a linear Schr\"odinger equation with a time-dependent harmonic
potential, and an initial Gaussian. It is well-known in the context of
the propagation of coherent states (see \cite{Hag80,CoRoBook}) that
the evolution of a Gaussian wave packet under a time-dependent harmonic
oscillator is a Gaussian wave packet. Therefore, it is consistent to
look for a solution to \eqref{eq:logNLS} of this form. 
Notice in particular that
\begin{equation}\label{dispersivenormgaussian}
  \|u(t)\|_{L^p({\mathbb R}^d)} = \left(\frac{2\pi}{p}\right)^{d/(2p)}
  \frac{|b(t)|}{\left(\prod_{j=1}^d\RE a_j(t)\right)^{1/(2p)}} \, ,\quad  1\le p\le \infty \, .
\end{equation}
To prove Theorem~\ref{theo:gaussian} we therefore need to find the asymptotic behavior in time of~$b(t)$ and~$a_j(t)$.

\subsubsection{The ODEs}Plugging \eqref{eq:gauss} into \eqref{eq:logNLS}, we obtain
\begin{align*}
  i\dot b -i\sum_{j=1}^d \dot a_j \frac{x_j^2}{2}b
  -\sum_{j=1}^d\frac{a_jb}{2} +\sum_{j=1}^d a_j^2
  \frac{x_j^2}{2}b = \lambda \Big( \ln\left(|b|^2\right) -\sum_{j=1}^d\left(\RE a_j\right) x_j^2
  \Big)b \, . 
\end{align*}
Equating the constant in $x$ and the factors of $x_j^2$, we get
\begin{align}
  \label{eq:aj}
 & i\dot a_j -a_j^2=2\lambda \RE a_j\,  ,\quad a_{j\mid t=0} =a_{0j}\, ,\\
\label{eq:b}
& i\dot b -\sum_{j=1}^d\frac{a_jb}{2} = \lambda b \ln\left(|b|^2\right),\quad
  b_{\mid t=0}=b_0\, .
\end{align}
We can express the solution to \eqref{eq:b} directly as a function of
the $a_j$'s: indeed
\begin{equation*}
   b(t) = b_0 \exp \Big({-i\lambda t \ln\left(|b_0|^2\right)-\frac{i}{2}\sum_{j=1}^dA_j(t)
    -i\lambda \sum_{j=1}^d\IM\int_0^tA_j(s)ds} \Big)\, ,
\end{equation*}
where we have set
\begin{equation*}
  A_j(t):=\int_0^t a_j(s)ds \, .
\end{equation*}
We also infer from \eqref{eq:aj} that $y:=\RE a_j$ solves
$  \dot y = 2y\IM a_j,$ hence
\begin{equation*}
  \RE a_j(t) = \RE a_{0j} \exp \Big({2\int_0^t\IM a_j(s)ds}\Big) \, . 
\end{equation*}
Since the equations \eqref{eq:aj} are decoupled as $j$ varies, we
simply consider from now on
\begin{equation}
  \label{eq:a}
  i\dot a -a^2=2\lambda \RE a\,  ,\quad a_{\mid t=0} =a_0=\alpha_0+i\beta_0\, ,
\end{equation}
which amounts to assuming $d=1$ in \eqref{eq:logNLS}. 
Following
\cite{LiWa06}, we seek $a$ of the form
\begin{equation*} 
  a = -i\frac{\dot \omega}{\omega}\, \cdotp
\end{equation*}
Then \eqref{eq:a} becomes
\begin{equation*}
  \ddot \omega = 2\lambda \omega\IM\frac{\dot \omega}{\omega}\, \cdotp
\end{equation*}
Introducing the polar decomposition $\omega=r e^{i\theta}$,
we get
\begin{equation*}
 \ddot r -(\dot \theta)^2r = 2\lambda r \dot \theta,\quad
\ddot \theta r +2\dot \theta \dot r=0 \, .
 \end{equation*}
Notice that
\[
  \dot \theta_{\mid t = 0} = \alpha_0 \, , \quad \left(\frac{\dot r}r\right)_{\mid t = 0} = -\beta_0 \, .
\]
We therefore have a degree of freedom to set $r(0) $, and we decide~$r(0)  =1$ so
\begin{equation*}
  \dot \theta(0)=\RE a_0=\alpha_0\, ,\quad \dot r(0)=-\IM a_0=-\beta_0\, .
\end{equation*}
The second equation yields
\begin{equation*}
 \frac{d}{dt}\left( r^2\dot \theta\right) = r\left( 2\dot r \dot
 \theta+r\ddot \theta\right)=0 \, , 
\end{equation*}
so $r^2\dot\theta$ is constant and we can express the problem in terms
of $r$ only:  we write
\begin{equation}\label{eq:quiesta}
  a (t)= \frac{\alpha_0}{r(t)^2}-i\frac{\dot r(t)}{r(t)} \, ,
\end{equation}
with
\begin{equation}\label{eq:rhopointpoint}
   \ddot r =\frac{\alpha_0^2}{r^3} + 2\lambda
   \frac{\alpha_0}{r} \,,\quad r(0)=1 \,,\ \, \dot r(0)= -\beta_0  \,. 
\end{equation}
Multiplying by $\dot r$ and integrating, we infer
\begin{equation}
  \label{eq:rhopointcarre}
  \left(\dot r\right)^2 = \beta_0^2 +\alpha_0^2\left(1-\frac{1}{r^2}\right)
  +4\lambda \alpha_0 \ln r \,. 
\end{equation}
Back to the solution $u$, in the case when~$d=1$ then writing in view of~(\ref{dispersivenormgaussian}) and~\eqref{eq:quiesta}
\begin{equation*}
  \|u(t)\|_{L^\infty({\mathbb R}^d)} = |b(t)| = |b_0| \exp \Big({\frac{1}{2}\int_0^t
    \IM a(s)ds} \Big)= \frac{|b_0|}{\sqrt{r(t)}} 
    \end{equation*}
we find that the study of~$r(t)$ is enough to find the dispersion rate
of~$u(t)$. Once the rate in one space dimension is known, the result
in~$d$ space dimensions follows directly. Moreover, we compute directly
  \[
  \begin{aligned}
  \|\nabla u(t)\|_{L^2({\mathbb R}^d)}^2  &=\frac12\pi^{d/2}
   \frac{|b(t)|^2}{\left(\prod_{j=1}^d\RE a_j(t)\right)^{1/2}} \sum_{j=1}^d\frac{|a_j(t)|^2}{\left(\RE
    a_j(t)\right) }\\
    & = \frac{\pi^{d/2}|b_0|^2}
    {2\left(\prod_{j=1}^d r_j(t)\right){\left(\prod_{j=1}^d\RE a_j(t)\right)^{1/2}} } \sum_{j=1}^d\frac{|a_j(t)|^2}{\left(\RE
    a_j(t)\right) }  \\
   & =   \frac{\pi^{d/2}|b_0|^2}{2
   \sqrt{\prod_{j=1}^d \alpha_{0j}
   } }  \sum_{j=1}^d  \Big((\dot r_j)^2 +\frac{\alpha_0^2}{r_j^2}\Big) \frac1 { \alpha_{0j}}\\
 &=  c+\lambda \frac{\pi^{d/2}|b_0|^2}{\sqrt{
   \prod_{j=1}^d \alpha_{0j}} }  \sum_{j=1}^d  \ln r_j(t) \, . 
 \end{aligned}
\]
As soon as~$r_j(t)\to \infty$ when~$|t |\to \infty$, the~$H^1$
norm therefore becomes unbounded. This is proved to be the case below
(with an explicit rate): actually it can be seen from the rate
provided in Lemma~\ref{eq:rhopointpoint}  below that the energy remains
bounded because the unbounded contributions of both parts of the
energy cancel exactly.  
\subsection{Study of $r(t)$}
\label{sec:dispersive}
The aim of this paragraph is to prove the following result. Recall notation~(\ref{eq:defell}).
\begin{lemma}\label{rhobehaviour}
Let~$r$ solve~{\rm(\ref{eq:rhopointpoint})}. Then 
as~$t \to \infty$, there holds
\[
 r(t)  =  2t \sqrt{\lambda \alpha_{0 } \ln t} \, \Big(1+{\mathcal O}\big(\ell(t)\big)\Big) \, .
\]
 \end{lemma}
 The proof of the lemma is achieved in three steps: first we prove, in Paragraph~\ref{sct:firststep}, that~$r(t) \to \infty$ as~$t \to \infty$. In view of that result   it is natural to approximate the solution to
\eqref{eq:rhopointpoint} by
\begin{equation}
  \label{eq:rhoeff}
    \ddot r_{\rm eff} =2\lambda
   \frac{\alpha_0}{ r_{\rm eff}} \, ,\quad  r_{\rm eff}(T)=r(T)\, ,\
   \dot  r_{\rm eff}(T)= \dot r(T) \, ,
\end{equation}
for $T\gg 1$. This is proved in  Paragraph~\ref{sct:secondstep}, along with a first estimate on the large time behavior of~$ r_{\rm eff}$. The conclusion of the proof is achieved in  Paragraph~\ref{sct:thirdstep}, by proving Lemma~\ref{lem:tau}.
 \subsubsection{First step: $r(t) \to \infty$}\label{sct:firststep}
We readily see from \eqref{eq:rhopointcarre}
that $r$ is bounded from below:
\begin{equation*}
  r(t)\ge \exp\left(-\frac{\beta_0^2+\alpha_0^2}{4\lambda \alpha_0}\right)>0\, ,\quad \forall t\in {\mathbb R}\, .
\end{equation*}
Now let us prove that $r(t)\to +\infty$ as $t\to +\infty$. 
Assume first that $\dot r(0)>0$. Then~\eqref{eq:rhopointpoint} yields~$\ddot r\ge
0$, hence $\dot r(t)\ge \dot r(0)$ for all $t\ge 0$, and 
\begin{equation}\label{rholargerthant}
  r(t)\ge \dot r(0)t+1\Tend t {+\infty} +\infty\, .
\end{equation}
On the other hand, for $\dot r(0)\le 0$, assume that $r$ is
bounded, $r(t)\le 
M$. Then  \eqref{eq:rhopointpoint} yields
\begin{equation*}
  \ddot r(t)\ge \frac{\alpha_0^2}{M^3} + 2\lambda \frac{\alpha_0}{M}\, ,
\end{equation*}
hence a contradiction for $t$ large enough. We infer that for $T$
sufficiently large, there holds~$r(T)\ge 1$ and $\dot r(T)>0$. The first
case then implies~$r(t)\to +\infty$.  

\medskip

Note that we have proved in particular that
\begin{equation}\label{eq:largetimer}
\exists T \ge 1 \, , \quad \dot r(T) >0 \quad \mbox{and} \quad \forall t \ge T \, , \quad r(t) \ge  \dot r(T) (t-T) + 1 \, .
\end{equation}

\subsubsection{Second step: $r(t)\sim r_{\rm{eff}}(t)$ with a rough bound}\label{sct:secondstep}
\begin{lemma}\label{rhoeffbehaviour}
There is~$T $ large enough so that defining~$r_{\rm eff} $ the solution of~{\rm(\ref{eq:rhoeff})} then as~$t \to \infty$, there holds
\[
 |r_{\rm eff}(t)| =   2 t\sqrt{\lambda \alpha_0 \ln t
} +\epsilon (t \sqrt {\ln t}
) \,, \quad \mbox{and} \quad 
|r(t) -  r_{\rm eff}(t)| \le C(T)   t \, , \quad \forall t \ge T 
\, ,
\]
where~$\epsilon(t)/t$ goes to zero as~$t$ goes to infinity. \end{lemma}
\begin{proof}
Let us start by studying~$ r_{\rm eff}$. Multiplying~(\ref{eq:rhoeff}) by~$\dot r_{\rm eff}$ and integrating, we get
\begin{align*}
  \left(\dot r_{\rm eff} (t)\right)^2 &= \left(\dot r(T)\right)^2 + 4\lambda
  \alpha_0\ln r_{\rm eff} (t) - 4\lambda
  \alpha_0\ln r(T) \\
&=  4\lambda
  \alpha_0\ln r_{\rm eff} (t) +\beta_0^2 +\alpha_0^2\left(1-\frac{1}{r(T)^2}\right)\, ,
\end{align*}
where we have used \eqref{eq:rhopointcarre} at time $t=T$. Denote by
\begin{equation*}
  C_0 := \beta_0^2 +\alpha_0^2\left(1-\frac{1}{r(T)^2}\right)\approx \beta_0^2
  +\alpha_0^2=|a_0|^2\, , 
\end{equation*}
since $T\gg 1$. By similar arguments as in the proof of~(\ref{eq:largetimer}) in
Paragraph~\ref{sct:firststep}, we have $\dot r_{\rm
  eff}(t)>0$ for all~$t \ge T$, and
  \[
   r_{\rm
  eff}(t) \ge \dot r (T) (t-T)+1
  \]
  hence
\begin{equation*}
  \dot r_{\rm eff} (t) =\sqrt{4\lambda
  \alpha_0\ln r_{\rm eff} (t) +C_0} \, . 
\end{equation*}
Separating the variables,
\begin{equation*}
  \frac{dr_{\rm eff}}{\sqrt{4\lambda
  \alpha_0\ln r_{\rm eff}  +C_0}}=dt \, ,
\end{equation*}
so we naturally consider the anti-derivative
\begin{equation*}
  I:= \int \frac{dr}{\sqrt{4\lambda
  \alpha_0\ln r  +C_0}}\, \cdotp
\end{equation*}
The change of variable 
\begin{equation*}
  y := \sqrt{4\lambda  \alpha_0\ln r +C_0}
\end{equation*}
yields
\begin{equation*}
  I = \frac{1}{2\lambda \alpha_0}\int e^{(y^2-C_0)/(4\lambda \alpha_0)}dy\, .
\end{equation*}
Since for~$x$
 large (Dawson function, see e.g. \cite{AbSt64}),
 \begin{equation*}
  \int e^{x^2} dx\sim \frac{1}{2x }e^{x^2}\,,
\end{equation*}
we infer
\begin{equation*}
  I\sim \frac{r} {\sqrt{4\lambda
  \alpha_0\ln r  +C_0}}\,\cdotp
\end{equation*}
In particular,
\begin{equation*}
 \frac{ r_{\rm eff} (t)}{\sqrt{4\lambda
  \alpha_0\ln r_{\rm eff} (t) +C_0}}\Eq t {+\infty}  t\,,
\end{equation*}
hence
\begin{equation*}
  \frac{r_{\rm eff} (t)}{\sqrt{\ln r_{\rm eff}(t)}}\Eq t {+\infty}
    2t\sqrt{\lambda \alpha_0}\,. 
\end{equation*}
We conclude that
\begin{equation*}
  r_{\rm eff} (t) \Eq t {+\infty}2 {t}\sqrt{\lambda
    \alpha_0\ln t}  \, . 
\end{equation*} 
Now let us prove that~$r$ can be well approximated by~$r_{\rm eff}$. We define~$h:= r-r_{\rm eff}$ and we want to prove that if~$T$ is chosen large enough, then~$h(t) \lesssim t$ when~$t \to \infty$.
We have
\[
\begin{aligned}
\dot h(t) &= \sqrt{4\lambda
  \alpha_0\ln r(t) +\beta_0^2 +\alpha_0^2\Big(1-\frac{1}{r(t) ^2}\Big)} \\
  & \quad -\sqrt{4\lambda
  \alpha_0\ln r_{\rm eff} (t) +\beta_0^2 +\alpha_0^2\Big(1-\frac{1}{r(T)^2}\Big)} \\
  & \le  \sqrt{4\lambda
  \alpha_0 \Big | \ln \frac{r(t) }{r_{\rm eff} (t)}  \Big | + \alpha_0^2\Big(\frac{1}{r(T)^2} - \frac{1}{r(t) ^2}\Big)} \, .
\end{aligned}
\]
Given~$\varepsilon \in (0,1/2)$, let~$T\ge 1$ be large enough so that for all~$t \ge T$
\begin{equation}\label{condition0}
 r_{\rm eff} (t) \ge   {t}\sqrt{\lambda
    \alpha_0\ln t}\end{equation}
and
\begin{equation}\label{condition1}
\displaystyle  \alpha_0^2\Big(\frac{1}{r(T)^2} - \frac{1}{r(t) ^2}\Big) \le \varepsilon^2 \, .
\end{equation}
We shall also need that
\begin{equation}\label{condition2}
\displaystyle \left(   2 \frac{ (\lambda
  \alpha_0)^{1/4} }{\sqrt{\ln T}}  + \varepsilon  \right) \le \frac12 \, \cdot
\end{equation}
 Then noticing that
\[
 \Big | \ln \frac{r(t) }{r_{\rm eff} (t)} \Big |  = \Big |  \ln \Big( 1+ \frac{h(t) }{r_{\rm eff} (t)}\Big) \Big |  
 \le  \frac{|h(t)| }{r_{\rm eff} (t)} 
 \le  \frac{|h(t)| }{ {t}\sqrt{\lambda
    \alpha_0\ln t}}  \le  \frac{|h(t)| }{ {t}\sqrt{\lambda
    \alpha_0\ln T}} 
\]
as soon as~$t \ge T$ thanks to~(\ref{condition0}),
we infer that
\[
\forall t \ge T \, ,\quad \dot h(t)  \le \varepsilon + 2 \sqrt{\lambda
  \alpha_0} \left( \frac{|h(t)| }{ {t}\sqrt{\lambda
    \alpha_0\ln T}}\right)^{1/2} \, , \quad \mbox{with} \quad h(T) = 0 \, .
\]
Our goal is to prove that the function~$t \mapsto h(t)/t$ is bounded for large~$t$, so let~$T^*>T$ be the maximal time such that
\[
\forall t \in [T,T^*) \, , \quad |h(t)| \le t \,  .
\]
Then for~$ t \in [T,T^*)$,  
\[
 \begin{aligned}
 \dot h(t)  \le \varepsilon + 2  (\lambda
  \alpha_0)^{1/4} \frac{1}  { \sqrt{   \ln T } }
\end{aligned}
 \]
 so  thanks to~(\ref{condition2})
 \[
  h (t)   \le \Big( \varepsilon + 2  (\lambda
  \alpha_0)^{1/4} \frac{1}  { \sqrt{  \ln T } } \Big) (t-T) \le \frac t 2\, ,
 \]
  which contradicts the maximality of~$T^*$. 
   The result follows, and Lemma~\ref{rhoeffbehaviour} is proved.
    \end{proof}
 \subsubsection{Third step: $r(t)\sim r_{\rm{eff}}(t)$ with improved bound}\label{sct:thirdstep}
Let us end the proof of Lemma~\ref{rhobehaviour}. By~(\ref{eq:rhopointcarre}) and as in the previous paragraph, we have  for $T$ sufficiently large so that $\dot r(t)\ge
\dot r(T)>0$  for $t\ge T$:
\begin{equation*}
  \dot r =
  \sqrt{C_0+\alpha_0^2\left(\frac{1}{r(T)^2}-\frac{1}{r^2}\right)
    +4\lambda \alpha_0 \ln r} \, ,
\end{equation*}
  \emph{with the same constant $C_0$ as above}:
recall that
\begin{equation*}
  \dot r_{\rm eff}  =
  \sqrt{C_0+4\lambda \alpha_0 \ln r_{\rm eff} } \, . 
\end{equation*}
To lighten notation let us  recall that~$h:=r-r_{\rm eff}$ and let us define
\[
R_{\rm eff}:=C_0+4\lambda \alpha_0 \ln r_{\rm eff}\, .
\]
Then using  a Taylor expansion for  $\dot r$, we have:
\[
\begin{aligned}
  \dot r&  =
  \sqrt{R_{\rm eff}+\alpha_0^2\big(\frac{1}{r(T)^2}-\frac{1}{r^2}\big)
   +4\lambda \alpha_0 \ln
            \big(1+\frac{h}{r_{\rm eff}}\big)} \\
             & =  \sqrt{R_{\rm eff}}
                    \sqrt{   1+ \frac{1}{R_{\rm eff}
 }\alpha_0^2 \big(\frac{1}{r(T)^2}-
\frac{1}{r^2} \big) +4 \frac{ \lambda \alpha_0 }{R_{\rm eff}}\ln       \big(1+\frac{h}{r_{\rm  eff}} \big)  }  \, .
           \end{aligned}
\]
On the one hand we know that~$R_{\rm eff} \to \infty$ and by Lemma~\ref{rhoeffbehaviour} we have~$h \lesssim t$ and~$  r_{\rm  eff} \Eq t \infty {t}\sqrt{ \ln t}  $
so we infer that
\[
 \dot r\Eq t \infty  \sqrt{R_{\rm eff} }
 \Big(
 1+
\frac{1}{2R_{\rm eff} }
\Big(
\alpha_0^2
\big(\frac{1}{r(T)^2}-\frac{1}{r^2}\big) +4\lambda \alpha_0 \ln \big(1+\frac{h}{r_{\rm  eff}}\big)\Big)
 \Big) \, .
\]
As a consequence
\[
\dot r- \dot r_{\rm eff}    \Eq t \infty
  \frac{1}{2\sqrt{R_{\rm eff}} }
\Big(
\alpha_0^2\big(\frac{1}{r(T)^2}-
\frac{1}{r^2}\big)
 +4\lambda \alpha_0 \ln \big(1+\frac{h}{r_{\rm
                                     eff}}
 \big)
\Big)
\]
and since~$ {h}/{r_{\rm
                                     eff}} = {\mathcal O}(1/\sqrt{\ln t})$
we infer that
\[
 \dot r- \dot r_{\rm eff} \Eq t \infty \frac{C(T)}{ \sqrt{\lambda \ln t}} \, \cdotp
\]
By integration, and comparison of diverging integrals, we find
\begin{equation*}
   h  (t)\Eq t \infty C_1 \frac{t}{\sqrt{\ln
        t}} \, ,
\end{equation*}
hence
\begin{equation*}
  r(t) = 2t\sqrt{\lambda \alpha_0\ln t}\,  \Big(1 +{\mathcal O}\big(\ell (t)\big)\Big) \, , 
\end{equation*}
as soon as we know that this holds for $r_{\rm eff}$.  Lemma~\ref{rhobehaviour} is therefore proved, up to the   study of the universal dispersion $\tau$.  

\subsection{Study of the universal dispersion $\tau(t)$: proof of Lemma~\ref{lem:tau}}
It remains to   prove
Lemma~\ref{lem:tau}.  By scaling, we may assume $\lambda =1$, to lighten the notations. 
Introduce the approximate solution
\begin{equation*}
  \tau_{\rm eff}(t) :=2t\sqrt{  \ln t}. 
\end{equation*}
We have clearly
\[
\sqrt {\ln t} = \sqrt {\ln   \tau_{\rm eff}} \left(1  + {\mathcal O}\left(\frac{\ln \ln t}{\ln
    \tau_{\rm eff}}\right) \right)\, . 
\]
 In view of a comparison with \eqref{eq:tau}, which reads
 \[
 \dot\tau = 2 \sqrt {\ln \tau} \, , 
 \]
  write
\[
\begin{aligned}
  \dot   \tau_{\rm eff} &= 2\sqrt{\ln t}+\frac{1}{\sqrt{\ln t}}  
 = 2\sqrt{\ln   \tau_{\rm eff} }
 \left( 
 1+ {\mathcal O} \left( \frac{\ln \ln t}{ \ln   \tau_{\rm eff}}\right)
 \right)   
  = 2\sqrt{\ln   \tau_{\rm eff} } + {\mathcal O} \left( \frac{\ln \ln t}{\sqrt{ \ln t}}\right) \, .
\end{aligned}
\]
Thus,
\begin{align*}
  \dot \tau-\dot   \tau_{\rm eff} & = 2\left(\sqrt{\ln \tau}-\sqrt{\ln   \tau_{\rm eff}}\right) +
                          {\mathcal O} \left( \frac{\ln \ln t}{\sqrt{ \ln t}}\right)\\ 
& = 2\sqrt{\ln   \tau_{\rm eff} +\ln \frac{\tau}{  \tau_{\rm eff}}} - 2\sqrt{\ln   \tau_{\rm eff}}  
+ {\mathcal O} \left( \frac{\ln \ln t}{\sqrt{ \ln t}}\right) \, .
\end{align*}
Since we already know from Lemma~\ref{rhoeffbehaviour} that
$\tau/  \tau_{\rm eff}\to 1$, we obtain
\begin{equation*}
  \dot \tau-\dot   \tau_{\rm eff}  =  {\mathcal O} \left( \frac{\ln \ln t}{\sqrt{ \ln t}}\right) \, , \quad
  \text{and}\quad 
 \tau-  \tau_{\rm eff} =  {\mathcal O} \left(t \frac{\ln \ln t}{\sqrt{ \ln t}}\right) \, , 
\end{equation*}
by integration. This proves Lemma~\ref{lem:tau}. \qed
\smallbreak

Back to the previous section, we simply note that
\begin{equation*}
  \dot r_{\rm eff} -\sqrt{\alpha_0}\dot \tau= \sqrt{C_0+4\lambda
    \alpha_0 \ln r_{\rm eff} } - \sqrt{4\lambda
    \alpha_0 \ln \tau } \, ,
\end{equation*}
with $C_0\not =0$ in general, so the same computation as above yields
\begin{equation*}
   \dot r_{\rm eff} -\sqrt{\alpha_0}\dot \tau=  {\mathcal O}\left(\frac{1}{\sqrt{\ln
     r_{\rm eff}}}\right)={\mathcal O}\left(\frac{1}{\sqrt{\ln t}}\right),
\end{equation*}
hence
\begin{equation*}
   r_{\rm eff} -\sqrt{\alpha_0} \tau =  {\mathcal O} \left( \frac{t}{\sqrt{
      \ln t}}\right) \, , 
\end{equation*}
by integration. This completes the proof of Lemma~\ref{rhobehaviour}. \qed

\section{General case: preparation for the proof of Theorem~\ref{theo:main}}
\label{sec:preparation}
In this section we prove~(\ref{eq:apv}) and~(\ref{eq:moments}) of
Theorem~\ref{theo:main}. 
\subsection{First a priori estimates}
Recall that by definition, $v$ is related to $u$ through the relation
\begin{equation}\label{eq:vu}
 u(t,x)
  =\frac{1}{\tau(t)^{d/2}}v\left(t,\frac{x}{\tau(t)}\right)
\frac{\|u_0\|_{L^2({\mathbb R}^d)}}{\|\gamma\|_{L^2({\mathbb R}^d)}} 
\exp\Big({i\frac{\dot\tau(t)}{\tau(t)}\frac{|x|^2}{2}}\Big) \, ,
\end{equation}
where $\tau$ is the solution to 
\begin{equation*}
  \ddot \tau = \frac{2\lambda }{\tau} \, ,\quad \tau(0)=1 \, ,\quad \dot
  \tau(0)=0 \, .
\end{equation*}
Then $v$ solves
\begin{equation*}
 i{\partial}_t v +\frac{1}{2\tau(t)^2}\Delta_y  v = \lambda  v\ln\left\lvert
    \frac{v}{\gamma}\right\rvert^2-\lambda d v\ln \tau
+2\lambda
  v\ln\left(\frac{\|u_0\|_{L^2({\mathbb R}^d)}}{\|\gamma\|_{L^2({\mathbb R}^d)}} \right)
  \, , 
\end{equation*}
where we recall that $\gamma(y)=e^{-|y|^2/2}$, and the initial datum
for $v$ is
\begin{equation*}
v_{\mid t=0}=v_0:=\frac{\|\gamma\|_{L^2({\mathbb R}^d)}}{\|u_0\|_{L^2({\mathbb R}^d)}} u_0 .
\end{equation*}
 Using the gauge
transform consisting in replacing~$v$ with~$ve^{-i\theta(t)}$ for
\[
\theta(t) =
\lambda d\int_0^t\ln \tau(s)ds -2\lambda
 t \ln(\|u_0\|_{L^2}/\|\gamma\|_{L^2}),
\]
 we may assume that the
last two terms are absent, and we focus our attention on
\begin{equation}
  \label{eq:v}
  i{\partial}_t v +\frac{1}{2\tau(t)^2}\Delta_y  v = \lambda v\ln\left\lvert
    \frac{v}{\gamma}\right\rvert^2,\quad 
v_{\mid t=0}=v_0 \, .
\end{equation}
We compute
\begin{equation*}
 \mathcal E (t):= \IM\int_{{\mathbb R}^d} \bar v(t,y){\partial}_t v(t,y)dy= \mathcal
 E_{\rm kin} (t)+\lambda \mathcal E_{\rm ent}(t) \, ,
\end{equation*}
where
\begin{equation*}
  \mathcal E_{\rm kin}(t):= \frac{1}{2\tau(t)^2}\|\nabla_y v(t)\|_{L^2}^2  
\end{equation*}
is the (modified) kinetic energy and
\begin{equation*}
 \mathcal  E_{\rm ent} (t):= \int_{{\mathbb R}^d} |v(t,y)|^2 \ln\left\lvert
    \frac{v(t,y)}{\gamma(y)}\right\rvert^2dy
\end{equation*}
is a relative entropy.  The transform~\eqref{eq:vu} is   unitary on~$L^2({\mathbb R}^d)$ so the
conservation of mass for $u$ trivially corresponds to the conservation
of mass for $v$: 
\begin{equation}\label{eq:massconservationv}
\|v(t)\| _{L^2}=\|v_0\|_{L^2}=\|\gamma\|_{L^2} \, .
\end{equation} 
The Csisz\'ar-Kullback inequality reads (see
e.g. \cite[Th.~8.2.7]{LogSob}), for $f,g\ge 0$ with $\int f=\int g$,
\begin{equation*}
  \|f-g\|_{L^1({\mathbb R}^d)}^2\le 2 \|f\|_{L^1({\mathbb R}^d)}\int f(x)\ln \left(\frac{f(x)}{g(x)}\right) dx.
\end{equation*}
Thanks to~\eqref{eq:massconservationv},
this  inequality  yields
\begin{equation}\label{eq:CK1}
  \mathcal E_{\rm ent}(t)\ge\frac{1}{2\|\gamma^2\|_{L^1}}\left\| |v(t)|^2-\gamma^2\right\|_{L^1({\mathbb R}^d)}^2 \, ,
\end{equation}
hence in particular $\mathcal E_{\rm ent}\ge 0$. 
 We easily compute
\begin{equation}\label{eq:E}
  \dot {\mathcal  E} = -2\frac{\dot \tau}{\tau}\mathcal E_{\rm kin} \, .
\end{equation}
Ideally, we would like to prove directly $\mathcal E(t)\Tend t \infty 0$. The property
$\mathcal E(t)\to 0$ can be understood as follows: 
\begin{itemize}
\item $\mathcal E_{\rm kin}\to 0$ means that $v$ oscillates in space
  more slowly than 
  $\tau$, hence that the main spatial oscillations of $u$ have been taken into account
    in \eqref{eq:vu} (as a matter of fact, the boundedness of
    $\mathcal E_{\rm
    kin}$ suffices to reach this conclusion).
\item $\mathcal E_{\rm ent}\to 0$ implies $|v(t)|^2\to
  \gamma^2$ strongly in $L^1({\mathbb R}^d)$. 
\end{itemize}
It turns out than in the case of Gaussian initial data, we can infer
from Section~\ref{sec:gaussian} that indeed $\mathcal E(t)\to 0$, each term
going to zero logarithmically in time (see Section~\ref{sec:entropie}
for the case of $\mathcal E_{\rm ent}$). In the
general case, we cannot reach this conclusion. Note however that if we
had $\mathcal E_{\rm kin}\gtrsim 1$, then integrating \eqref{eq:E} we would
get $\mathcal E(t)\to -\infty$ as $t\to \infty$, hence a contradiction. Therefore,
\begin{equation*}
  \exists t_k\to \infty \, ,\quad \mathcal E_{\rm kin}(t_k)\to 0 \,.
\end{equation*}

We now prove the first part of Theorem~\ref{theo:main}, that is,
\eqref{eq:apv} which is recast and complemented in the next lemma.
\begin{lemma}\label{lem:apv}
 Under the assumptions of Theorem~{\rm\ref{theo:main}}, there holds
\begin{equation*}
 \sup_{t\ge 0}\left(\int_{{\mathbb R}^d}\left(1+|y|^2+\left|\ln
    |v(t,y)|^2\right|\right)|v(t,y)|^2dy +\frac{1}{\tau(t)^2}\|\nabla_y
  v(t)\|^2_{L^2({\mathbb R}^d)}\right)<\infty
\end{equation*} 
and
\begin{equation}\label{eq:integralkin}
 \int_0^\infty \frac {\dot \tau(t)}{\tau^3(t)}\|\nabla_y
  v(t)\|^2_{L^2({\mathbb R}^d)} dt<\infty.
\end{equation}
\end{lemma}
\begin{proof}
 Write
  \begin{equation*}
    \mathcal E_{\rm ent} = \int_{{\mathbb R}^d} |v|^2 \ln|v|^2+\int _{{\mathbb R}^d} |y|^2|v|^2 \,,
  \end{equation*}
and
\begin{equation*}
  \int_{{\mathbb R}^d} |v|^2 \ln|v|^2 = \int_{|v|>1} |v|^2 \ln|v|^2+\int_{|v|<1}|v|^2 \ln|v|^2.
\end{equation*}
We have
\begin{equation*}
  \mathcal E_+:= \mathcal E_{\rm kin}+\lambda \int_{|v|>1} |v|^2 \ln|v|^2 +\lambda \int  _{{\mathbb R}^d} |y|^2|v|^2 \le
   \mathcal E(0)+ \lambda\int_{|v|<1}|v|^2 \ln\frac{1}{|v|^2} \, \cdotp
\end{equation*}
The last term is controlled by
\begin{equation*}
   \int_{|v|<1}|v|^2 \ln\frac{1}{|v|^2}\lesssim \int _{{\mathbb R}^d} |v|^{2-\epsilon} \, ,
\end{equation*}
for all $\epsilon>0$. We conclude thanks to the estimate~\eqref{interpolationinequality} with~$\alpha = 1$:
\begin{equation*}
  \int_{{\mathbb R}^d} |v|^{2-\epsilon} \lesssim \|v\|_{L^2}^{2-(1+d/2)\epsilon}
  \|yv\|_{L^2}^{d \epsilon/2} \, ,
\end{equation*}
for $0<\epsilon< {4}/{(d+2)}$. This implies  
\begin{equation*}
  \mathcal E_+\lesssim 1 +\mathcal E_+^{d\epsilon/4} \, ,
\end{equation*}
and thus $\mathcal E_+\in L^\infty({\mathbb R})$. 

Finally, \eqref{eq:integralkin}  follows from
\eqref{eq:E}, since $\mathcal E(t)\ge 0$ for all $t\ge0$.
\end{proof}
\subsection{Convergence of some quadratic quantities}
Let us prove~(\ref{eq:moments}), as stated in the next lemma.
\begin{lemma}\label{lem:quad-v}
 Under the assumptions of Theorem~{\rm\ref{theo:main}}, there holds
\begin{equation*}
   \int_{{\mathbb R}^d}
  \begin{pmatrix}
    1\\
y\\
|y|^2
  \end{pmatrix}
|v(t,y)|^2dy\Tend t \infty 
 \int_{{\mathbb R}^d}
  \begin{pmatrix}
    1\\
y\\
|y|^2
  \end{pmatrix}
\gamma^2(y)dy \, .
\end{equation*}
\end{lemma}
\begin{proof}
Introduce
\begin{equation*}
  I_1(t) := \IM \int_{{\mathbb R}^d} \overline v (t,y)\nabla_y v(t,y)dy \, ,\quad I_2(t) :=
  \int_{{\mathbb R}^d} y|v(t,y)|^2dy \, . 
\end{equation*}
We compute:
\begin{equation}\label{eq:evolI1I2}
  \dot I_1= -2\lambda I_2 \,,\qquad \dot I_2 = \frac{1}{\tau^2(t)}I_1 \, .
\end{equation}
Set $\tilde I_2:=\tau I_2$: we have $\ddot {\tilde I}_2=0$, hence
(unless the data are well prepared in the sense that $I_1(0)=0$) 
\begin{equation*}
  I_2(t) = \frac{1}{\tau(t)}\left(\dot
 {\tilde I}_2(0)t+\tilde I_2(0)\right)=\frac{1}{\tau(t)}\left(-I_1(0)t+I_2(0)\right)\Eq t \infty 
  \frac{c}{\sqrt{\ln t}} \, ,
\end{equation*}
and
\begin{equation*}
  I_1(t)\Eq t \infty \tilde c\frac{t}{\sqrt{\ln t}}\, \cdotp
  \end{equation*}
In particular,
\begin{equation*}
   \int_{{\mathbb R}^d} y|v(t,y)|^2dy\Tend t \infty 0 = \int_{{\mathbb R}^d}y\gamma(y)^2dy \, .
\end{equation*}
In order to obtain estimates for higher order quadratic observables,
we observe that Cauchy-Schwarz inequality and Lemma~\ref{lem:apv}
  yield 
  \begin{equation*}
    \left|\IM \int v (t,y) y \cdot\nabla_y
    \bar v  (t,y) dy \right| \lesssim \tau(t). 
  \end{equation*}
We now go back to
the conservation of energy for~$u$,
 \begin{equation*}
    \frac{d}{dt}\left(\frac{1}{2}\|\nabla u(t)\|_{L^2}^2 +\lambda \int_{{\mathbb R}^d}
    |u(t,x)|^2\ln |u(t,x)|^2dx\right)=0 \, .
  \end{equation*}
 and translate this property into estimates on~$v$, recalling the mass
 conservation for~$v$ stated in~(\ref{eq:massconservationv}).  
\begin{align*}
  \frac{d}{dt}\Bigg(&
{\mathcal E}_{\rm kin}+\frac{(\dot \tau)^2}{2}\int |y|^2|v|^2-\frac{\dot
\tau}{\tau}\IM \int v (t,y) y \cdot\nabla_y
    \bar v  (t,y) dy +\lambda \int |v|^2\ln|v|^2\\
&  -\lambda d\ln \tau \int |v|^2
+2\lambda \|\gamma\|_{L^2}^2\ln\left(\frac{\|u_0\|_{L^2}}{\|\gamma\|_{L^2}}\right)\Bigg)=0 \, .
\end{align*}
In the above expression, all the terms are
bounded functions of time, but possibly three: from the above
estimate, the third term is ${\mathcal O}(\dot\tau)={\mathcal O}(\sqrt{\ln t})$, while the terms
\[
\frac{(\dot
  \tau)^2}{2}\int |y|^2|v|^2\quad \mbox{and} \quad -\lambda d\ln \tau \int |v|^2 
\]
are of the same order, ${\mathcal O}(\ln t)$.
We infer
\begin{equation*}
  \frac{(\dot \tau)^2}{2} \int |y|^2|v|^2 -\lambda d\ln \tau \int
  |v|^2={\mathcal O}\left(\sqrt{\ln t}\right) \,.
\end{equation*}
Integrating \eqref{eq:tau}, we find
\begin{equation*}
  \frac{(\dot \tau)^2}{2}= 2\lambda \ln\tau \,,
\end{equation*}
hence
\begin{equation*}
  \int |y|^2|v|^2 -\frac{d}{2}\|v\|_{L^2}^2={\mathcal O}\left(\frac{1}{\sqrt{\ln t}}\right)\,,
\end{equation*}
and the lemma is proved.
\end{proof}

At this stage, we therefore have proved Theorem~\ref{theo:main}, up to
the final point regarding the asymptotic profile for $|v|^2$.

\section{End of the proof of Theorem~\ref{theo:main}}
\label{sec:end}
In this section we conclude the proof of Theorem~\ref{theo:main} by
obtaining the weak convergence to a universal profile as stated
in~\eqref{eq:weaklimitv}. 
 
\subsection{Hydrodynamical approach}
\label{sec:hydro}
We recall that the Madelung transform is a classical tool (see
e.g. \cite{Madelung,LandauQ,GM97}, or the survey \cite{CaDaSa12}) to relate
the (nonlinear) Schr\"odinger equation to fluid dynamics equations, via
the change of unknown 
\begin{equation}\label{eq:vaphi}
v(t,y)  = a(t,y)e^{i\phi(t,y)} \, , \quad a, \phi \in {\mathbb R} \, .
\end{equation}
Formally one obtains in our case the system of equations
\begin{equation}\label{eq:madelung-vrai}
  \left\{
    \begin{aligned}
      &{\partial}_t \phi + \frac{1}{2\tau^2}|\nabla_y \phi|^2
         +\lambda \ln \left|
        \frac{a}{\gamma}\right|^2 =\frac{1}{2\tau^2}\frac{\Delta_y
        a}{a}  \\
& {\partial}_t a +\frac{1}{\tau^2}\nabla_y \phi\cdot\nabla_y a
+\frac{1}{2\tau^2}a\Delta_y
  \phi =0 \, ,
    \end{aligned}
\right.
\end{equation}
which is easily related to the compressible Euler equations by using
the change of unknown 
\begin{equation}\label{eq:vareuler}
\rho: = a^2  \, , \quad \Lambda := a \nabla \phi \, , \quad J :=
a \Lambda  \, . 
\end{equation}
Note that in the explicit case of Gaussian initial data studied in
Section~\ref{sec:gaussian}, $\rho$ is bounded in Sobolev spaces
uniformly in time, whereas $\Lambda$ and $J$ are unbounded as~$t\to
\infty$. 
In terms of these hydrodynamical variables, the above system becomes
\begin{equation}
  \label{eq:hydro-madelung}
  \left\{
\begin{aligned}
&{\partial}_t \rho + \frac{1}{\tau^2}\nabla\cdot J=0 \\
& {\partial}_t J +\frac{1}{\tau^2}\nabla\cdot \left(\Lambda \otimes
\Lambda\right)+\lambda \nabla \rho +2\lambda 
y \rho =\frac{1}{4\tau^2}\Delta \nabla\rho
-\frac{1}{\tau^2}\nabla\cdot \left(\nabla
\sqrt\rho\otimes \nabla \sqrt\rho\right) \\
& {\partial}_j J^k-{\partial}_k J^j =
2\Lambda^k{\partial}_j\sqrt\rho-2\Lambda^j{\partial}_k\sqrt\rho\, ,\quad j,k\in
\{1,\dots,d\} \, .
\end{aligned}
\right.
\end{equation}
Note that in the case where the initial data for
\eqref{eq:hydro-madelung} are well prepared, in the sense that they
stem from the polar decomposition of an initial wave function as in
\eqref{eq:vaphi}--\eqref{eq:vareuler}, then the approach
from \cite{AnMa09,AnMa12} (see also 
\cite[Section~5]{CaDaSa12}) can readily be adapted to show that
\eqref{eq:hydro-madelung} holds true in the distributional sense,
thanks to a suitable polar factorization technique. However, we will
see below that the Madelung transform can be by-passed
thanks to a more direct approach, where \eqref{eq:vaphi} is not invoked.
We shall prove that
\[
\rho (t) \mathop{\rightharpoonup}\limits_{t\to \infty} \gamma^2\quad  \text{weakly in }L^1({\mathbb R}^d) \, . 
\]
This will stem from the fact that the weak limit of $\rho$ evolves
according to a Fokker--Planck operator.  We note that 
a formal  link between the hydrodynamical formulation of
\eqref{eq:logNLS}  and the   Fokker--Planck equation can be found
in~\cite{LoMG09,GLM14}.  
\subsection{Heuristics}
\label{sec:heuristics}
Let us explain the heuristics of the proof, which will be made rigorous in the next section.  Formally only retaining the higher order terms (in terms of growth in time) in~\eqref{eq:hydro-madelung} we are led to studying the following simple model
\begin{equation}
  \label{eq:baby-madelung}
  \left\{
\begin{aligned}
&{\partial}_t \rho + \frac{1}{\tau^2}\nabla\cdot J=0 \\
& {\partial}_t J +\lambda \nabla \rho +2\lambda
y \rho =0 \, .
\end{aligned}
\right.
\end{equation}
Note that in the explicit case of the evolution of a Gaussian (recall the computations of Section~\ref{sec:gaussian}), we can
check that in the above simplification, we have indeed eliminated
negligible terms. 
By elimination of~$J$, \eqref{eq:baby-madelung} implies that
\begin{equation*}
{\partial}_t \left(\tau^2{\partial}_t \rho\right) =\lambda \nabla\cdot \left(\nabla+2y\right)\rho=
\lambda L\rho \, ,   
\end{equation*}
where 
\[
L:=\Delta_y + \nabla_y\cdot(2y \,  \cdot)
\]   
is a Fokker--Planck operator. On the other hand,
\begin{equation*}
 {\partial}_t \left(\tau^2{\partial}_t \rho\right) = \tau^2{\partial}_t^2\rho + 2\dot \tau \tau {\partial}_t \rho \, , 
\end{equation*}
so since $\tau^2\ll (\dot \tau \tau)^2$, 
it is natural to change scales in time and define~$s$ such that
\[
\frac{\dot \tau \tau}{\lambda}  {\partial}_t={\partial}_s \,, 
\]
or in other words define the following change of variables:
\begin{equation}\label{eq:defs}
  s = \int\frac{\lambda}{\dot \tau \tau}= \int \frac{\ddot \tau}{2
    \dot \tau} = \frac{1}{2}\ln \dot \tau(t) \, .
\end{equation}
Notice that
\begin{equation}\label{eq:sequivlnlnt}
  s\sim \frac{1}{4} \ln \ln t \, ,\quad t\to \infty \, .
\end{equation}
Then again discarding formally lower order terms we find 
\begin{equation*}
  {\partial}_s \rho = L\rho \, ,
\end{equation*}
for which it is well-known (see for instance~\cite{GaWa05}) that in
large times the solution converges strongly  to an element of the
kernel of~$L$, hence a Gaussian. Notice that the convergence is
exponentially fast in~$s$ variables, so returning to~$t$ variables
produces a logarithmic decay due to~\eqref{eq:sequivlnlnt}: we recover
the logarithmic convergence rate observed in the Gaussian case
(Section~\ref{sec:gaussian}). 

The difficulty to make this argument rigorous is the justification that the lower order terms may indeed be discarded, since 
we have very little control on higher norms on~$v$ to guarantee
compactness in space of the solution: we have more precisely a sharp
control of the momenta of $v$, but rather poor estimates in
$H^1$. More precisely, we do expect $v$ to oscillate rapidly in time
(in view of the Gaussian case), but $\sqrt\rho$ should be bounded in
$H^1$, a property that does not seem easy to prove (because of the
prefactor $1/\tau^2$ in the equation). This is the main obstacle to 
proving strong convergence to a Gaussian in the general case, and
explains why in the end we only obtain a weak convergence result
in~$L^1$. This is made precise in the next section. 

\subsection{End of the proof}
\label{sec:conclusion}
Let us follow the steps of the previous paragraph, this time
neglecting no term. First, we rewrite
  \eqref{eq:hydro-madelung} by using a more direct, and rigorous,
  derivation of a hydrodynamical system. Instead of the definition
\eqref{eq:vareuler}, relying on \eqref{eq:vaphi} (which in turn
demands a rigorous justification of the polar decomposition), we set
\begin{equation}
  \label{eq:decomp-hydro}
  \rho:=|v|^2,\quad J:=\IM\left(\bar v \nabla v\right).
\end{equation}
When \eqref{eq:vaphi} is available, this definition is equivalent to
\eqref{eq:vareuler}, but it has the advantage of being completely
rigorous for free. We then have
\begin{equation}
  \label{eq:hydro-madelung-bis}
  \left\{
\begin{aligned}
&{\partial}_t \rho + \frac{1}{\tau^2}\nabla\cdot J=0 \\
& {\partial}_t J +\lambda \nabla \rho +2\lambda 
y \rho =\frac{1}{4\tau^2}\Delta \nabla\rho
-\frac{1}{\tau^2}\nabla \cdot \RE\left(\nabla v \otimes \nabla \bar v\right) \, .
\end{aligned}
\right.
\end{equation}
By elimination of~$J$,  
\begin{equation*}
  {\partial}_t\left(\tau^2 {\partial}_t \rho\right)  = -{\partial}_t \nabla\cdot  J =  
\lambda L\rho -\frac{1}{4\tau^2}\Delta^2 \rho
  -\frac{1}{\tau^2}\nabla\cdot\left(  \nabla \cdot \RE\left(\nabla v \otimes
  \nabla \bar v\right) \right) \, ,
\end{equation*}
with again 
$
L:=\Delta + \nabla\cdot(2y \, \cdot).$ With the change of
variable~\eqref{eq:defs} we introduce the
 notation $\tilde
\rho(s(t),y):= \rho(t,y)$, and we find for~$\tilde \rho$ the following equation:
\begin{equation}
  \label{eq:dsrho}
  {\partial}_s \tilde \rho -\frac{2\lambda}{(\dot \tau)^2}{\partial}_s \tilde \rho
  +\frac{\lambda }{(\dot \tau)^2} {\partial}_s^2 \tilde \rho = L\tilde \rho
  -\frac{1}{4\lambda \tau^2}\Delta^2 \tilde \rho 
  -\frac{1}{\lambda \tau^2}\nabla\cdot\left(  \nabla \cdot \RE\left(\nabla
  \tilde v \otimes
  \nabla \bar{\tilde v}\right) \right)  \, ,
  \end{equation}
where one should keep in mind that the functions $\tau$ and $\dot
\tau$ also have undergone the change of time variable. In terms of
$s$, Lemma~\ref{lem:tau} yields
  \begin{equation*}
  \dot \tau(s) \Eq s \infty 2\sqrt \lambda e^{2s} \, ,\quad \tau(s)\Eq s
  \infty 2\sqrt \lambda e^{2s+e^{4s}} \, .
\end{equation*}
To make the discussion at the end of the previous subsection more
precise, and explain why we prove a weak convergence
  only, we comment on the various terms in \eqref{eq:dsrho}:
\begin{itemize}
\item The term $\frac{2\lambda}{(\dot \tau)^2}{\partial}_s \tilde \rho$ is
  essentially harmless in the large time limit, for it could be
  handled by a slight modification of the time variable, for
  instance.
\item The term $\frac{\lambda }{(\dot \tau)^2} {\partial}_s^2 \tilde \rho$ is
  expected to be negligible in the large time limit. However, it makes
  \eqref{eq:dsrho} second order in time: one would like to take
  advantage of the smoothing properties of $e^{sL}$, by using
  Duhamel's formula typically, but this approach is delicate in this
  context. Note that it has been established before that in similar
  situations, the parabolic behavior gives the leading order large
  time dynamics, even if the coefficient of ${\partial}_s^2\tilde \rho$ is not
  asymptotically vanishing (\cite{GaRa98}): the proof of this fact
  relies on energy estimates whose analogue in the case of
  \eqref{eq:dsrho} we could not establish.
\item By using the Fourier transform in space, it is easy to compute the
  fundamental solution of 
  \begin{equation*}
    {\partial}_s \tilde \rho= L\tilde \rho
  -\frac{1}{4\lambda \tau^2}\Delta^2 \tilde \rho\, ,
  \end{equation*}
in the same way as~\cite{GaWa02} (without the last term). 
\item A possible idea to prove that the solution to \eqref{eq:dsrho}
  converges (strongly) to the Gaussian $\gamma^2$ as $s\to \infty$
  would be to use the spectral decomposition of $L$, as given for
  instance in \cite{GaWa02}. The main issue is then that we can
  control the last term in \eqref{eq:dsrho} in $L^1$-based spaces, as
  opposed to $L^2$ where the spectral decomposition is available. Note
  that this term is the only one that prevents \eqref{eq:dsrho} from
  being a linear, homogeneous, equation. 
\end{itemize}
In terms of $s$, the time integrability property of $\mathcal E_{\rm
  kin}$ provided in~(\ref{eq:integralkin}) becomes
\begin{equation}\label{eq:int-nabla-v-s}
  \int_0^\infty\left(\frac{\dot \tau(s)}{\tau(s)}\right)^2 \|\nabla\tilde
  v(s)\|_{L^2}^2ds<\infty \, . 
\end{equation}
On the other hand, Lemma~\ref{lem:apv} yields
\begin{equation}\label{eq:borneap}
  \sup_{s\ge 0}\int_{{\mathbb R}^d} \tilde \rho(s,y)\left(1+|y|^2+|\ln \tilde
  \rho(s,y)|\right)dy<\infty \, . 
\end{equation}
Mimicking the general approach of e.g. \cite{De90,FePe99},
for $s\in [-1,2]$ and $s_n\to \infty$, set 
\begin{equation*}
  \tilde\rho_n(s,y) := \tilde \rho(s+s_n,y) \, .
\end{equation*}
From \eqref{eq:borneap} along with the de la Vall\'ee-Poussin   and
Dunford--Pettis Theorems, we get up to extracting a subsequence 
\begin{equation*}
  \tilde \rho_n\rightharpoonup \tilde \rho_\infty \quad \text{in
  }L^p_s(-1,2;L^1_y) \, , 
\end{equation*}
for all $p\in [1,\infty)$. Up to another subsequence,
\begin{equation*}
  \tilde \rho_n(0)\rightharpoonup \tilde \rho_{0,\infty} \quad \text{in
  }L^1_y \, . 
\end{equation*}
Moreover thanks  to \eqref{eq:borneap}    the family $\left(\tilde
\rho( s_n,\cdot\right))_n$ is tight and hence
\[
\int \tilde \rho_{0,\infty}  (y)\, dy = \int \gamma(y)^2 \, dy\,,
\]
and there holds also
\[
\int_{{\mathbb R}^d} \tilde \rho_{0,\infty} (y)\left(1+|y|^2+|\ln \tilde
 \rho_{0,\infty}(y)|\right)dy<\infty \, . 
\]
Property \eqref{eq:int-nabla-v-s} implies that 
  \begin{equation*}
    \frac{1}{\lambda \tau_n^2}\nabla\cdot\left(  \nabla \cdot \RE\left(\nabla
  \tilde v_n \otimes
  \nabla \bar{\tilde v}_n\right) \right) \rightharpoonup 0\text{ in }L^1_s
  W^{-2,1},\quad \text{where }\tau_n(s)=\tau(s+s_n). 
  \end{equation*}
In addition, in \eqref{eq:dsrho}, all the other terms but two
obviously go weakly to  zero,  which yields
\begin{equation}\label{eq:FKfinal}
  \left\{
      \begin{aligned}
        &{\partial}_s \tilde \rho_\infty  =L\tilde \rho_\infty\text{ in
        }\mathcal D '\left((-1,2)\times{\mathbb R}^d\right) ,\\
&\tilde \rho_{\infty\mid s=0} = \tilde \rho_{0,\infty}\in L^1 \, .
      \end{aligned}
\right.
\end{equation}
Thanks to the above bounds on~$\tilde \rho_{0,\infty} $ it is known
(see for instance~\cite[Corollary 2.17]{AMTU}) that the solution~$
\tilde \rho_\infty$ to~(\ref{eq:FKfinal}) is actually defined for all
$s\ge 0$ and 
satisfies
\begin{equation}\label{eq:FKfinallargetime}
\lim_{s\to \infty} \| \tilde \rho_\infty (s) - \gamma^2\|_{L^1({\mathbb R}^d)} = 0\,.
\end{equation}
We now go back to \eqref{eq:hydro-madelung-bis} and show that $\tilde
\rho_\infty $ is independent of $s$. In the $s$ variable, we have
\begin{equation}
  \label{eq:madelung-s}
  \left\{
    \begin{aligned}
      &{\partial}_s \tilde \rho +\frac{\dot \tau}{\lambda \tau}\nabla \cdot \tilde
      J=0 \\
& {\partial}_s \tilde J +\tau\dot \tau \left(\nabla +2y\right)\tilde \rho-\frac{\dot
  \tau}{4\lambda \tau}\nabla\Delta\tilde \rho = -\frac{\dot \tau}{\lambda
  \tau}\nabla \cdot \RE\left(\nabla
  \tilde v \otimes
  \nabla \bar{\tilde v}\right) \, . 
    \end{aligned}
\right.
\end{equation}
Since $J=\IM \bar v\nabla_yv$, \eqref{eq:int-nabla-v-s} implies
\begin{equation*}
  \frac{\dot \tau}{\tau}\tilde J\in L^2_sL^1_y \, .
\end{equation*}
With $\tilde
J_n(s) := \tilde J (s+s_n)$, we have
\begin{equation*}
\frac{\dot \tau}{\tau} \nabla\cdot \tilde  J_n \Tend n\infty
0\quad\text{in } L^2(-1,2;W^{-1,1}) \, , 
\end{equation*}
hence
\begin{equation}\label{eq:stat}
  {\partial}_s \tilde \rho_{\infty}=0 \, . 
\end{equation}
Putting \eqref{eq:FKfinallargetime} and \eqref{eq:stat} together, we infer that
$\tilde \rho_\infty=  \gamma^2$. The limit being
unique, no extraction of a subsequence is needed, and we conclude
\begin{equation*}
  \tilde\rho(s) \mathop{\rightharpoonup}\limits_{s\to \infty} \gamma^2 \quad
  \text{weakly in }L^1({\mathbb R}^d) \, .
\end{equation*}
Theorem~\ref{theo:main} is proved. \qed

\section{Proof of the corollaries}
\label{sec:entropie}

\subsection{Proof of Corollary~\ref{cor:Hs}}

In the energy for $u$,  write  the
potential energy in terms of $v$.
\begin{align*}
  \int_{{\mathbb R}^d}|u(t,x)|^2\ln|u(t,x)|^2dx & = -d\ln
  \tau(t)
\frac{\|u_0\|_{L^2}^2}{\|\gamma\|_{L^2}^2}
\int_{{\mathbb R}^d}|v(t,y)|^2dy\\
&\quad
 +\frac{\|u_0\|_{L^2}^2}{\|\gamma\|_{L^2}^2}
\int_{{\mathbb R}^d}|v(t,y)|^2\ln\left( \frac{\|u_0\|_{L^2}^2}{\|\gamma\|_{L^2}^2}|v(t,y)|^2\right)dy \\
&=-d \|u_0\|_{L^2}^2\ln
  \tau(t) +{\mathcal O}(1),
\end{align*}
from \eqref{eq:apv}. Therefore, the conservation of the energy for $u$
yields
\begin{equation*}
  \|\nabla u(t)\|_{L^2}^2 = 2E_0 +2\lambda d \|u_0\|_{L^2}^2\ln
  \tau(t) +{\mathcal O}(1)\Eq t \infty 2\lambda d \|u_0\|_{L^2}^2
  \ln t,
\end{equation*}
hence the first point of Corollary~\ref{cor:Hs}.
Now fix $0<s<1$. By interpolation, we readily have
\begin{equation*}
  \|u(t)\|_{\dot H^s}\lesssim \|u(t)\|_{L^2}^{1-s}\|u(t)\|_{\dot
    H^1}^s\lesssim \left(\ln t\right)^{s/2},
\end{equation*}
where we have used the conservation of mass and the above asymptotics.
 \smallbreak

The convergence in the Wasserstein distance $W_2$
(Remark~\ref{rem:wasserstein}) implies (see
e.g. \cite[Theorem~7.12]{Vi03}) 
\begin{equation}\label{eq:moment-int}
  \int |y|^{2s}|v(t,y)|^2 dy\Tend t \infty \int
  |y|^{2s}\gamma^2(y)dy. 
\end{equation}
The idea is then to apply a fractional derivative to \eqref{eq:uv},
that is
\begin{equation*}
  u(t,x)
  =\frac{1}{\tau(t)^{d/2}}v\left(t,\frac{x}{\tau(t)}\right)
\frac{\|u_0\|_{L^2}}{\|\gamma\|_{L^2}}
\exp \Big({i\frac{\dot\tau(t)}{\tau(t)}\frac{|x|^2}{2}} \Big) \, . 
\end{equation*}
In order to shortcut this step, we recall a lemma employed in a
somehow similar situation, even though in the context of
semi-classical limit. We therefore simplify the initial statement and
leave out the dependence on the semi-classical parameter:
\begin{lemma}[Lemma~5.1 from \cite{ACMA}]\label{lem:ACMA}
  There exists $C$ such that if $u\in H^1({\mathbb R}^d)$ and  $w$ is such that
  $\nabla w\in L^\infty({\mathbb R}^d)$,
  \begin{equation*}
  \||w|^s u\|_{L^2} \le \|u\|_{\dot H^s} + \|(\nabla -iw)u\|_{L^2}^s
    \|u\|_{L^2}^{1-s} + C \left(1+\|\nabla w\|_{L^\infty}\right)\|u\|_{L^2}. 
  \end{equation*}
\end{lemma}
In \cite{ACMA}, $w$ corresponds to the gradient of rapid oscillations
carried by an 
exponential, so we naturally introduce
\begin{equation*}
  w(t,x) = \frac{\dot\tau(t)}{\tau(t)}x.
\end{equation*}
In the present framework, Lemma~\ref{lem:ACMA} yields:
\begin{equation*}
 (\dot \tau)^s\||y|^sv(t)\|_{L^2}  
\frac{\|u_0\|_{L^2}}{\|\gamma\|_{L^2}}
\le \|u(t)\|_{\dot H^s} +
  \left\|\frac{1}{\tau}\nabla v(t)\right\|_{L^2}^s
\frac{\|u_0\|_{L^2}}{\|\gamma\|_{L^2}^s}
  + C\left(1 + \frac{\dot \tau}{\tau}\right)\|u_0\|_{L^2}.
\end{equation*}
The result follows readily: the behavior of the left hand side
is given by Lemma~\ref{lem:tau} and \eqref{eq:moment-int}, and all the
terms of the right hand side 
are bounded, but the first one. 
\subsection{Proof of Corollary~\ref{cor:entropie-gauss}}

In view of the tensorization in Theorem~\ref{theo:gaussian}, we prove
Corollary~\ref{cor:entropie-gauss} in the case~$d=1$ to lighten the
  notations, and we assume
\begin{equation*}
  u_0(x)=  b_0\exp \Big ({-a_0(x-x_0)^2/2} \Big) \, ,
\end{equation*}
with $b_0,a_0\in {\mathbb C}$, $\RE a_0=\alpha_0>0$. We start with an initial
center $x_0$ to show that in terms of $v$, the center is eventually
zero (like in \cite{GaWa05}). Recall that we have
\begin{equation*}
    u(t,x) = b_0\frac{1}{\sqrt{r(t)}} e^{i\phi(t)}
\exp \Big ({-\alpha_{0} \frac{(x-x_0)^2}{2r^2(t)}+i\frac{\dot r(t) }{r(t)}\frac{(x-x_0)^2}{2}} \Big ) \, ,
\end{equation*} 
with  $r$ solution to \eqref{eq:rhopointpoint}, $r(0)=1$, $\dot
r(0)=-\IM a_0$. We thus have, 
since
\begin{equation*}
  \|u_0\|_{L^2} = |b_0|\left(\frac{\pi}{\alpha_0}\right)^{1/4}\text{ and
  }\|\gamma\|_{L^2}= \pi^{1/4}\, ,
\end{equation*}
\begin{equation*}
\begin{aligned}
  v(t,y) =
  \frac{b_0}{|b_0|}\alpha_0^{1/4}
\sqrt{\frac{\tau(t)}{r(t)}}e^{i\phi(t)}
&\exp \Big ({-\alpha_0\frac{\tau^2}{r^2}\frac{y^2}{2}+\alpha_0\frac{\tau}{r^2}yx_0-
\alpha_0\frac{x_0^2}{2r^2}} \Big) \\
& \qquad \times \exp \Big( {i\left(\frac{\dot r}{r}-\frac{\dot
  \tau}{\tau}\right) \tau^2\frac{y^2}{2}-i\frac{\dot r}{r} \tau y
x_0+i\frac{\dot r}{r} \frac{x_0^2}{2} } \Big) \, .
\end{aligned}
\end{equation*}
In particular,
\begin{equation*}
  |v(t,y)|^2 = \frac{\tau(t)}{r(t)}\sqrt{\alpha_0}
\exp \Big ( {-\alpha_0\frac{\tau^2}{r^2}y^2
+2\alpha_0\frac{\tau}{r^2}yx_0-
\alpha_0\frac{x_0^2}{r^2}} \Big) \, .
\end{equation*}
Therefore, the relative entropy is 
\begin{align*}
  \mathcal E_{\rm ent}(t)& = \int_{{\mathbb R}} |v(t,y)|^2\ln \left(\frac{
    |v(t,y)|^2}{\gamma^2(y)}\right)dy \\
& = \ln\left(\sqrt{\alpha_0} \frac{\tau(t)}{r(t)}\right)\|v_0\|_{L^2}^2-
\left(\alpha_0\frac{\tau(t)^2}{r(t)^2} - 1\right)\int_{{\mathbb R}} 
 y^2 |v(t,y)|^2 dy\\
&\quad +2\alpha_0 x_0\frac{\tau(t)}{r^2(t)} \int_{{\mathbb R}}  y |v(t,y)|^2dy -
  \alpha_0\frac{x_0^2}{r^2(t)} \|v_0\|_{L^2}^2\Tend t \infty 0 \, ,
\end{align*}
where we have used the properties of the solutions to
\eqref{eq:rhopointpoint} and \eqref{eq:tau}, established in
Section~\ref{sec:gaussian}. The end of the corollary simply stems from
the  Csisz\'ar-Kullback inequality \eqref{eq:CK1}.

\section{Stability of the dynamics under a power-like perturbation}
\label{sec:power}
In this section, we give the main arguments to adapt the previous
proofs in the case of \eqref{eq:power},
\begin{equation*}
   i{\partial}_t u + \frac{1}{2}\Delta u = \lambda u \ln(|u|^2)+
  \mu|u|^{2\sigma}u\,,\quad u_{\mid t=0}=u_0 \, .
\end{equation*}
\subsection{Construction of the solution}
To show that \eqref{eq:power} has a solution $u\in L^\infty_{\rm
  loc}({\mathbb R};\Sigma)$ whose mass, angular momentum and energy as
conserved, we can follow exactly the same strategy as in
Subsection~\ref{sec:existence}, and consider the family of functions~$u_\varepsilon$ solution to 
\begin{equation*}
   i{\partial}_t u_\varepsilon + \frac{1}{2}\Delta u_\varepsilon = \lambda u_\varepsilon \ln(\varepsilon+|u_\varepsilon|^2)+
  \mu|u_\varepsilon|^{2\sigma}u_\varepsilon,\quad u_{\varepsilon\mid t=0}=u_0.
\end{equation*}
For fixed $\varepsilon>0$, the above equation has a unique, local, solution
$u\in C([-T,T];\Sigma)\cap
L^{\frac{4\sigma+4}{d\sigma}}([-T,T];L^{2\sigma+2})$, thanks to
Strichartz estimates (see 
e.g. \cite{CazCourant,Ginibre}). To see that the solution is global in
time for fixed $\varepsilon>0$ and satisfies estimates which are uniform in
  $\varepsilon$, we notice that the momentum
$I_{\varepsilon,1}=\int\left\langle x\right\rangle^2|u_\varepsilon(t,x)|^2dx$ is controlled exactly like
in Subsection~\ref{sec:existence}. On the other hand, we compute
\begin{align*}
  \frac{d}{dt}&\left(\frac{1}{2}\|\nabla u_\varepsilon\|_{L^2}^2+\lambda \int
  |u_\varepsilon|^2\ln \left(\varepsilon+|u_\varepsilon|^2\right)
  +\frac{\mu}{\sigma+1}\int|u_\varepsilon|^{2\sigma+2}\right) \\
&\qquad= \lambda \int
  \frac{|u_\varepsilon|^2}{\varepsilon+|u_\varepsilon|^2}{\partial}_t\left(|u_\varepsilon|^2\right),
\end{align*}
and since ${\partial}_t\left(|u_\varepsilon|^2\right)=-\IM(\bar u\Delta u)$, Gronwall's lemma yields
estimates which are uniform in $\varepsilon>0$, as desired, and we conclude
like in Subsection~\ref{sec:existence}.
\smallbreak
We note that the uniqueness argument used in the case $\mu=0$ can be
repeated provided that $u(t,\cdot)\in L^\infty({\mathbb R}^d)$, which is
granted if $d=1$ by Sobolev embedding. We cannot conclude ``as usual'' thanks
to Strichartz estimates (case $\lambda=0$), since the logarithmic
nonlinearity is not Lipschitz, and so uniqueness is not clear for~$d\ge 2$. 
\subsection{Estimates for $v$}
We note that in the presence of a power-like nonlinearity, we can no
longer rely on explicit computations in the case of Gaussian initial
data. This is not a problem, since the main role of the explicit
computations was to suggest the change of unknown function
\begin{equation*}
 u(t,x)
  =\frac{1}{\tau(t)^{d/2}}v\left(t,\frac{x}{\tau(t)}\right)
\frac{\|u_0\|_{L^2({\mathbb R}^d)}}{\|\gamma\|_{L^2({\mathbb R}^d)}} 
\exp \Big({i\frac{\dot\tau(t)}{\tau(t)}\frac{|x|^2}{2}} \Big) .
\end{equation*}
Now $v$ solves
\begin{equation*}
 i{\partial}_t v +\frac{1}{2\tau(t)^2}\Delta_y  v = \lambda v\ln\left\lvert
    \frac{v}{\gamma}\right\rvert^2-\lambda d v\ln \tau
+2\lambda
  v\ln\left(\frac{\|u_0\|_{L^2({\mathbb R}^d)}}{\|\gamma\|_{L^2({\mathbb R}^d)}} \right)+\frac{\tilde
  \mu}{\tau^{d\sigma} }|v|^{2\sigma}v
  \, , 
\end{equation*}
with 
\[
\tilde \mu := \left(\frac{\|u_0\|_{L^2({\mathbb R}^d)}}{\|\gamma\|_{L^2({\mathbb R}^d)}}
\right)^{2\sigma}\mu\,.
\]
We simply note that $\tilde \mu>0$, and the numerical value of $\mu$
is unimportant. Like before, the previous two terms can be absorbed by a
gauge transform, and we consider
\begin{equation*}
  i{\partial}_t v +\frac{1}{2\tau(t)^2}\Delta_y  v = \lambda v\ln\left\lvert
    \frac{v}{\gamma}\right\rvert^2+\frac{\tilde
  \mu}{\tau^{d\sigma}} |v|^{2\sigma}v\,,\quad v_{\mid t=0}=v_0\,. 
\end{equation*}
Noticing that the pseudo-energy
\begin{equation*}
  \mathcal E(t)= \frac{1}{2\tau(t)^2}\|\nabla_y v(t)\|_{L^2}^2 +\lambda \int
  |v|^2\ln \left|\frac{v(t,y)}{\gamma(y)}\right|^2dy+\frac{\tilde
    \mu}{(\sigma+1)\tau(t)^{d\sigma}}\int |v(t,y)|^{2\sigma+2}dy
\end{equation*}
satisfies
\begin{equation*}
  \dot{\mathcal E}= -2\frac{\dot
    \tau}{\tau}\times\frac{1}{2\tau(t)^2}\|\nabla_y v(t)\|_{L^2}^2 -
  d\sigma \frac{\dot
    \tau}{\tau}\times \frac{\tilde
    \mu}{(\sigma+1)\tau(t)^{d\sigma}}\|v(t)\|^{2\sigma+2}_{L^{2\sigma+2}}\,,
\end{equation*}
Lemma~\ref{lem:apv} is readily generalized to this case, with in
addition
\begin{equation*}
  \sup_{t\ge 0}\frac{1}{\tau(t)^{d\sigma}}\| v(t)\|^{2\sigma+2}_{L^{2\sigma+2}}<\infty\,,
\end{equation*}
and
property \eqref{eq:integralkin} is enriched with the extra estimate
\begin{equation}\label{eq:integralpot}
  \int_0^\infty\frac{\dot
    \tau(t)}{\tau(t)^{d\sigma+1}}\|v(t)\|_{L^{2\sigma+2}}^{2\sigma+2}dt<\infty\,. 
\end{equation}
The convergence of the momenta (Lemma~\ref{lem:quad-v}) is adapted in
a straightforward manner as well: the system \eqref{eq:evolI1I2} which
yields the convergence of the first momentum remains the same, and the
convergence of the momentum of order two follows from the same
argument. 
\subsection{Hydrodynamics}
To prove the weak convergence of $|v|^2$ to $\gamma^2$, we mimic the
proof presented in Section~\ref{sec:end}. Essentially, one new term
appears in \eqref{eq:hydro-madelung-bis}, 
\begin{equation*}
  \left\{
\begin{aligned}
&{\partial}_t \rho + \frac{1}{\tau^2}\nabla\cdot J=0 \\
& {\partial}_t J +\lambda \nabla \rho +2\lambda 
y \rho =\frac{1}{4\tau^2}\Delta \nabla\rho
-\frac{1}{\tau^2} =\nabla \cdot \RE (\nabla v\otimes
  \nabla \bar v)= -\frac{\tilde
  \mu}{\tau^{d\sigma}}\frac{\sigma}{\sigma+1}\nabla\left( \rho^{\sigma+1}\right)\, .
\end{aligned}
\right.
\end{equation*}
Working with the time variable $s$ defined in \eqref{eq:defs}, the
last two terms vanish in the large time weak limit: the term in
$|\nabla v|^2$ thanks to \eqref{eq:integralkin} (like before), and the
last term thanks to \eqref{eq:integralpot} (in the same fashion). 
\subsection{Growth of Sobolev norms}
To prove the final part of Theorem~\ref{theo:power}, we can resume
exactly the same arguments as in the proof of Corollary~\ref{cor:Hs}. 
\bibliographystyle{siam}

\bibliography{biblio}

\end{document}